\documentclass[12pt]{article}

\usepackage{color}
\usepackage{amsthm}
\usepackage{amsmath}
\usepackage{amssymb}
\usepackage[margin=.75in]{geometry}
\usepackage{enumerate}
\usepackage{hyperref}
\hypersetup{
  colorlinks   = true, 
  urlcolor     = blue, 
  linkcolor    = blue, 
  citecolor   = blue 
}

\usepackage{mathrsfs}

\title{Absolute continuity of the stable foliation of a Banach space mapping}

\theoremstyle{theorem}
\newtheorem{thm}{Theorem}[section]
\newtheorem*{thm*}{Theorem 1}
\newtheorem{cor}[thm]{Corollary}
\newtheorem*{cor*}{Corollary 1}
\newtheorem{lem}[thm]{Lemma}
\newtheorem{prop}[thm]{Proposition}

\theoremstyle{definition}

\newtheorem{defn}[thm]{Definition}
\newtheorem{rmk}[thm]{Remark}

\newtheorem{cla}[thm]{Claim}

\newcommand{\N}{\mathbb{N}}

\newcommand{\R}{\mathbb{R}}

\newcommand{\Z}{\mathbb{Z}}

\newcommand{\diam}{\operatorname{diam}}

\renewcommand{\a}{\alpha}

\renewcommand{\d}{\delta}
\newcommand{\e}{\epsilon}
\renewcommand{\l}{\lambda}
\newcommand{\graph}{\operatorname{graph}}
\newcommand{\lip}{\operatorname{Lip}}
\newcommand{\Id}{\operatorname{Id}}

\newcommand{\Bc}{\mathcal{B}}

\renewcommand{\tilde}{\widetilde}

\newcommand{\tf}{\tilde{f}}

\newcommand{\Sc}{\mathcal{S}}

\newcommand{\Oc}{\mathcal{O}}
\newcommand{\Wc}{\mathcal{W}}

\newcommand{\As}{\mathscr{A}}

\newcommand{\la}{\lambda}

\newcommand{\Lip}{\operatorname{Lip}}

\newcommand{\Dom}{\operatorname{Dom}}

\newcommand{\Gap}{\operatorname{Gap}}

\newcommand{\Nk}{\mathfrak N}

\title{Absolute continuity of stable foliations for mappings of Banach spaces}
\author{Alex Blumenthal\thanks{Courant Institute of Math. Sciences, New York University, New York, USA.  Email: alex@cims.nyu.edu.}
\and Lai-Sang Young\thanks{Courant Institute of Math. Sciences, New York University, New York, USA.  Email: lsy@cims.nyu.edu. This research was supported in part by NSF Grant DMS-1363161.}
}

\begin{document}

\maketitle

\abstract{
We prove the absolute continuity of stable foliations for mappings of Banach
spaces satisfying conditions consistent with time-$t$ maps of 
certain classes of dissipative PDEs. 
This property is crucial for passing information from submanifolds transversal to the
stable foliation to the rest of the phase space; it is also used in proofs of ergodicity. 
Absolute continuity of stable foliations is well known in finite dimensional 
hyperbolic theory. On Banach spaces, the
absence of nice geometric properties poses some additional difficulties.}

\section{Introduction and Setting}

In finite dimensional dynamical systems theory, positive Lebesgue
or Riemannian measure sets have often been equated with observable events,
and the absolute continuity of stable foliations has been 
 a very useful tool for connecting positive measure sets on unstable 
manifolds to positive measure sets in the phase space. Here we have
assumed that the phase space supports a meaningful notion of volume, e.g.,
it is a Riemannian manifold, and the measures in question are associated with
volumes or induced volumes on unstable manifolds. The connection above has made it possible 
for dissipative systems with chaotic attractors to have
a natural notion of {\it physically relevant} invariant measures. 
Indeed one of the most important advances in finite dimensional hyperbolic theory 
in the last half century is the idea of SRB measures, which govern the distributions of 
positive Lebesgue measure sets of initial conditions thanks to the absolute
continuity of stable foliations 
(see e.g. \cite{eckmannRuelle,pughShub,youngDiffSurv}). An equally important use
of this property is in proofs of ergodicity, 
via the well known argument of Hopf \cite{hopf}. This argument has been 
used many times: { we mention applications to geodesic flows on manifolds 
of negative curvature (see, e.g., \cite{anosovAC}) and to dispersing billiards (e.g., \cite{sinaiAC, liveraniAC}); see also \cite{pesin, pughShub}.}

In infinite dimensional dynamical systems, such as those on Banach spaces,
there is no natural notion of volume, hence no obvious concept of 
``observable events"; yet the idea of what constitutes a ``typical solution"
for a PDE seems no less important. It is in the context of attempting to 
offer an answer to these questions
that the idea of absolute continuity of stable foliations appears. Using 
Haar measure to define a notion of ``positivity of measure" on finite 
dimensional subspaces or submanifolds of Banach spaces, it has been shown 
that for a Banach space system with a center manifold, there 
is a strong stable foliation that is absolutely continuous \cite{LYZinertialAC}. 
{ Via this strong stable foliation, properties that are determined by asymptotic
future orbit distributions are passed from the center manifold to the rest of the
phase space, and the absolute continuity of this foliation enables us to
define a notion of ``typical initial condition", a notion of ``almost everywhere" in Banach spaces, that is dynamically connected to volumes on center manifolds.}

In this paper, we extend the idea of absolute continuity of stable foliations 
to dissipative dynamical systems with quasi-compact derivative operators on 
Banach spaces without assuming the existence of center manifolds.
We state and carry out in detail a complete proof of this result for strong stable manifolds
of nonuniformly hyperbolic dynamical systems. As a corollary, we show that
{ the basins of SRB measures with nonzero Lyapunov exponents 
are ``visible", in the sense that for many families of initial conditions smoothly 
parametrized by $[0,1]^k$, orbits starting from a positive Lebesgue measure subset
are described by SRB measures.}

There are several proofs of absolute continuity in finite dimensions, a 
testimony to the centrality of this result in the subject. Our proof follows in outline 
the one sketched in \cite{youngDiffSurv}, and is different than \cite{pughShub, pesin, katokStrelcyn}. 
We mention that \cite{katokStrelcyn}, as well as the very recent paper \cite{LLLhilbert}, 
both prove a similar result for mappings of Hilbert spaces. 
An important difference between Hilbert and Banach spaces is that 
the latter need not have good geometry. Any proof of absolute continuity 
hinges on (i) the action of holonomy maps 
{(defined by sliding along stable manifolds)} on 
balls or objects with nice geometric shapes, and 
(ii) covering lemmas on transversals by objects of the same kind. 
In this paper, we have had to devise ways to overcome the difficulty that
Banach space balls are not necessarily nice. We believe our proof is to-the-point and 
{ concise}, perhaps one 
of the most direct even among finite dimensional proofs. We have also 
included a complete proof of the formula for the Radon-Nikodym derivatives 
of holonomy maps, a fact often claimed without proof in papers in finite 
dimensions.


\bigskip
{ The setting of this paper is as follows: } 
Let $\Bc$ be a Banach space with norm $|\cdot|$. We consider $(f, \mu)$, where $f : \Bc \to \Bc$ is a map and $\mu$ is an $f$-invariant Borel probability measure. We assume: 
\begin{itemize} 
\item[(H1)] (i) $f$ is injective and $C^{2}$ Fr\'echet differentiable; 

\vspace{-3 pt}
(ii) the derivative of $f$ at $x \in \Bc$, denoted $df_x$, is also injective.

\vspace{-3 pt}
\item[(H2)] (i) $f$ leaves invariant a compact set $\As \subset \Bc$, with $f(\As)=\As$;

\vspace{-4 pt}
(ii) $\mu$ is supported on $\As$. 

\vspace{-3 pt}
\item[(H3)] We assume 
\[l_\alpha(x):=\lim_{n\to\infty}\frac1n \log |df^n_x|_\alpha \ < \ 0 
\quad \mbox{ for } \ \mu-\text{a.e. } x\ . \]
Here $|df^n_x|_\alpha$ is the Kuratowski measure of noncompactness of 
the set $df^n_x(B)$, where
$B$ is the unit ball in $\Bc$ (see, e.g., \cite{nussbaum} for properties of $|\cdot|_\a$).
\end{itemize}

{
Conditions (H1), (H2)(i) and (H3) are known to hold for systems defined by large 
classes of dissipative PDEs (see \cite{henry}); the compact set $\As$ is often an attractor.
The existence of invariant measures on $\As$ is not an additional assumption; such
measures always exist.}

\medskip
To motivate the material in Sections 2 and 3, we first state a rough version 
of one of our main results, containing yet-to-be-defined terms.

\bigskip \noindent
{\bf Provisional Theorem} \ {\it Let $W^{ss}$ be a strong stable foliation of $f$. We assume $W^{ss}$
has codimension $k \in \mathbb Z^+$, and let $\Sigma^1, \Sigma^2$ be two 
embedded $k$-dimensional
disks in $\Bc$, close to one another and roughly parallel, both transversal to $W^{ss}$. 
We assume that the holonomy map 
$$p: \check \Sigma^1 \to \Sigma^2$$ 
from $\Sigma^1$ to $\Sigma^2$ along $W^{ss}$ is defined on 
$\check \Sigma^1 \subset \Sigma^1$, i.e. for $x \in \check \Sigma^1$, $p(x)$ is the unique point in 
$W^{ss}_{{\rm loc}, x} \cap \Sigma^2$, where $W^{ss}_{{\rm loc}, x}$ is a local $W^{ss}$-manifold at $x$. Then $p$ is absolutely continuous, in the sense that if $B \subset
\Sigma^2$ is a Borel set such that $\nu_{\Sigma^2}(B)=0$, then 
$\nu_{\Sigma^1}(p^{-1}B)=0$. Here $\nu_\Sigma$ is the induced volume on
an embedded disk $\Sigma$.} 

\bigskip

Induced volumes and other preliminaries are given in Section 2. More technical
preparation, including the strong stable foliation and transversals, are discussed 
in Section 3. The {Provisional Theorem} above is formulated precisely as {\bf Theorem A}
and proved in Section 4. {\bf Theorem B}, which gives precise Radon-Nikodym 
derivatives of holonomy maps, is stated and proved in Section 5. Section 6 contains
some consequences of these results for SRB measures with no zero Lyapunov exponents,
including {\bf Theorem C}, on ergodic decomposition, and {\bf Theorem D},
on the ``visibility" of SRB measures.

\section{Preliminaries}

\subsection{Banach space geometry}\label{subsec:bSpaceGeometryAC}

First we explain what is meant by induced volume in the statement of the {\bf Provisional Theorem}. 

\begin{defn}
Let $E \subset \Bc$ be a finite-dimensional subspace. We define the \emph{induced volume }$m_E$ on $E$ to be the unique Haar measure on $E$ for which
$$
m_E \{u \in E \mid |u| \leq 1\} = \omega_k
$$
where $k = \dim E$ and $\omega_k$ is the Lebesgue measure of the Euclidean unit ball in $\R^k$.
\end{defn}

Once volumes are defined, the notion of {\it determinant} follows naturally:
Let $A: \Bc \to \Bc$ be a bounded operator, and let $E \subset \Bc$ be a subspace
of finite dimension. Let $B_E$ denote the closed unit ball in $E$. Then
\begin{align*}
\det(A |E ) = \begin{cases} \frac{m_{AE}(A B_E) }{m_E(B_E)} & A|_E \text{ injects} \\ 0 & \text{else.} \end{cases} \, .
\end{align*}

The notion of induced volume above is defined for one subspace at a time.
For it to be useful, it is necessary to ensure some regularity as subspaces are varied.
The Hausdorff distance between two closed subspaces $E, E' \subset \Bc$ is 
defined to be
$$
d_H(E, E') = \max \{ \sup\{d(e, S_{E'}) : e \in S_E \}, \sup\{d(e', S_E) : e' \in S_{E'}\}\}
$$
where $S_E =\{v \in E \,|\, |v|=1\}$.

\begin{prop}[\cite{BYentropy}, Proposition 2.15] \label{prop:detReg} 
For any $k \geq 1$ and any $M > 1$ there exist $L, \e > 0$ with the following properties. If $A_1, A_2 : \Bc \to \Bc$ are bounded linear operators and $E_1, E_2 \subset \Bc$ are $k$-dimensional subspaces for which
\begin{gather*}
|A_j|,~ |(A_j|_{E_j})^{-1}| \leq M \quad j=1,  2 \, ,\\
|A_1 - A_2|, ~d_H(E_1, E_2) \leq \e \, ,
\end{gather*}
then we have the estimate
\begin{equation} \label{regularity}
\left| \log \frac{\det(A_1 | E_1)}{\det (A_2 | E_2)} \right| \leq L (|A_1 - A_2| + d_H(E_1, E_2))\ .
\end{equation}
\end{prop}

\begin{rmk}[see \cite{BYentropy}] \label{rmk:constantsDependence}
Later, when we apply Proposition \ref{prop:detReg} to distortion estimates, we will need to use the dependence of the constants $\e, L$ on the parameters $k, M$. One can show that there exists a constant $C_k \geq 1$, depending only on the dimension $k \in \N$, such that we may take $\e = (C_k M^{10 k })^{-1}$ and $L = C_k M^{10 k }$ in the conclusion to Proposition \ref{prop:detReg}.
\end{rmk}

Treating induced volumes on finite dimensional linear subspaces as volume elements, one obtains by the usual construction
a notion of induced volume $\nu_W$ on a finite dimensional submanifold $W$ (see, e.g., Sect. 5.3 in \cite{BYentropy}).
This is the measure on transversals used in the statement of the {\bf Provisional Theorem} in Section 1.

\bigskip

For computations, it is often convenient to work with the \emph{gap} $\Gap(E, E')$, defined by
\[
\Gap(E, E') = \sup_{v \in S_E} d(v, E') \, .
\]
The quantities $\Gap$ and $d_H$ are related as follows:

\begin{lem}[\cite{kato}]\label{lem:symmetryCloseness}
For all closed subspaces $E, E'$, we have
$$ d_H(E, E') \leq \max\{\Gap(E, E'), \Gap(E', E) \} \leq 2 d_H(E, E')\ .  
$$
If additionally $E, E' \subset \Bc$ are closed subspaces with the same finite codimension $q$, then
\[
\Gap(E', E) \leq \frac{q \Gap(E, E')}{1 - q \Gap(E, E')} \, ,
\]
so long as the denominator in the above expression is $> 0$.
\end{lem}

For a more complete discussion of results on Banach space geometry, 
induced volumes and determinants etc. that are relevant for the extension
of finite dimensional ergodic theory to Banach space maps, see \cite{BYentropy}, Section 2.

\subsection{Multiplicative Ergodic Theorem (MET)} 

To fix notation, we recall the following version of the MET, which is adequate for our purposes:  Let $X$ be a compact metric space, and let $f : X \to X$ be a homeomorphism
preserving a Borel probability measure $\mu$ on $X$. 
We consider a continuous map $T : X \to {\bf B}(\Bc)$ where
${\bf B}(\Bc)$ denotes the space of bounded linear operators on $\Bc$,
the topology on ${\bf B}(\Bc)$ being the   operator norm topology.
We assume additionally that $T_x:= T(x)$ is injective for every $x \in X$, and write 
$T_x^n = T_{f^{n-1} x} \circ \cdots \circ T_x$. Define
\[
l_\a(x) = \lim_{n \to \infty} \frac1n \log |T^n_x|_\a 
\]
for $\mu$-almost every $x \in X$ (as in (H3) in Section 1). 

\newcommand{\oGamma}{\overline{\Gamma}}

\begin{thm}[Multiplicative ergodic theorem \cite{thieu}]  \label{thm:MET} 
Under the hypotheses above, for any measurable function $\l_\a : X \to \R$ for which $\l_{\a} > l_{\a}$ $\mu$-almost surely, 
there is a measurable, $f$-invariant set $\oGamma \subset X$ 
with $\mu(\oGamma)=1$, a measurable function $r : \oGamma \to \Z_{\geq 0}$, and on the level
sets of $r$ a collection of measurable functions $\l_1, \cdots, \l_{r(x)} : X \to \R$ such that
\[\l_1(x) > \l_2(x) > \cdots > \l_{r(x)}(x) > \l_\a(x) \, , \] 
for which the following properties hold. 
For any $x \in \oGamma$, there is a splitting 
\[\Bc = E_1(x) \oplus E_2(x) \oplus \cdots \oplus E_{r(x)}(x) \oplus F(x)\]
such that
\begin{itemize}
\item[(a)] for each $i=1,2,\dots, r(x)$, $\dim E_i(x) = m_i(x)$ is finite, $T_x E_i(x) = E_i(f x)$,
 and for any $v \in E_i(x) \setminus \{0\}$, we have
$$
\l_i(x) = \lim_{n \to \infty} \frac{1}{n} \log |T^n_x v| = - \lim_{n \to \infty} \frac{1}{n} \log  |( T^n_{f^{-n} x})^{-1} v |\ ;
$$
\item[(b)] the distribution $F$ is closed and finite-codimensional, satisfies $T_x F(x) \subset F(f x)$ and
$$
\la_\a(x) \geq \limsup_{n \to \infty} \frac{1}{n} \log |T^n_x |_{F(x)}|  \ ;
$$
\item[(c)] the mappings $x \mapsto E_i(x), x \mapsto F(x)$ are $\mu$-continuous along the level sets of $r$ (see Definition \ref{defn:muCont} below), and
\item[(d)] writing $\pi_i(x)$ for the projection of $\Bc$ onto $E_i(x)$ via the splitting at $x$, 
we have
$$
\lim_{n \to \pm \infty} \frac{1}{n} \log |\pi_i(f^n x)| = 0 \quad a. s.
$$
\end{itemize}
\end{thm}

\begin{defn}\label{defn:muCont}
Let $X$ be a compact metric space and $\mu$ a Borel probability on $X$, and let $Z$ be a metric space. We say that a mapping $\Psi : X \to Z$ is $\mu$-continuous if there is an increasing sequence of compact subsets $\bar K_1 \subset \bar K_2 \subset \cdots \subset X$ with the properties that (i) $\Psi|_{\bar K_n}$ is a continuous mapping for all $n$ and (ii) $\mu(\cup_n \bar K_n) = 1$.
\end{defn}

\noindent For related facts on $\mu$-continuity, see \cite{BYentropy}, Section 3. 

\begin{rmk}\label{rmk:cutoffFunction}
The function $\la_\a$ appearing in Theorem \ref{thm:MET} should be thought of as 
mitigating a \emph{cutoff}, { prescribed in advance,} for the Lyapunov
spectrum of $(f, \mu; T)$. In the case where $(f, \mu)$ is ergodic, $l_\a $ is constant almost surely, and so for all purposes it
 suffices to apply the MET with $\la_\a$ equal to any constant strictly greater than $l_\a$. 
 When $(f, \mu)$ is not ergodic, { $l_\a$ is a measurable
 function taking values in $[-\infty, 0)$ (see (H3)), and it may be natural, even necesary,
for $\la_\a$ to be nonconstant. Given $l_\a$, an example of $\la_\a$ may be as follows:
For arbitrarily fixed constants $\gamma \in (0,1)$ and $\hat \la_\a \in (-\infty, 0)$, define
\begin{align}\label{eq:cutoffFunctionAC}
\lambda_\a(x)  = 
\begin{cases}
(1 - \gamma) l_\a(x)   & l_\a(x) > - \infty \\
\hat \la_\a & l_\a(x) = - \infty\ .
\end{cases} 
\end{align}
Observe that the function $\la_\a$ so defined has the property that $l_\a < \la_\a <0$\ ; it 
converges to
 $l_\a$ as $\gamma \to 0$ and $\hat \la_\a \to -\infty$, and importantly, 
it is an $f$-invariant function.}
\end{rmk}

{
\subsubsection*{Invariant sets defined by splitting of the Lyapunov spectrum} 


It is convenient to represent 
$\bar \Gamma$ as a countable union of 
positive $\mu$-measure invariant subsets on which certain quantities in the MET have
uniform bounds. Here is one way to systematically enumerate such a collection of 
invariant sets:

For $\la^* \in \R$, $m, p \in \Z_{> 0}$, define
\begin{align}\label{eq:gammaDecomp1AC}\begin{split}
\Gamma(\la^*; m, p)  =  \big\{x \in \oGamma : \, & \la_\a (x) < \la^* - \frac1p, \ 
\min_i |\la_i(x) - \la^*| > \frac1p,\ \dim E^+_x = m\big\}\ .
\end{split}\end{align}
When $\la_\a$ is $f$-invariant, each $\Gamma(\la^*; m, p)$ is invariant under $f$, and that the countable union
\begin{align*}
\bigcup_{\substack{m, p \in \mathbb Z_{> 0} \\ \la^* \in \mathbb Q}} \Gamma(\la^*; m, p) 
\end{align*}
is a full $\mu$-measure set. On $\Gamma(\la^*; m, p)$, we have the following splitting:
Let $\Bc_x$ denote the tangent space at $x \in \Gamma(\la^*; m, p)$. Then
$\Bc_x = E^+_x \oplus E^-_x$, where
$E^\pm_x$ are defined by
\begin{align}\label{defnEplus}
E^+_x = \bigoplus_{i : \la_i(x) > \la^*} E_i(x) \,  \quad \text{ and } \quad E^-_x = F(x) \oplus\bigg(\bigoplus_{i : \la_i(x) < \la^*} E_i(x) \bigg)\, .
\end{align}
Thus dim$(E^+) = m$ and $df_x(E^+_x) = E^+_{fx}$, while dim$(E^-) = \infty$ 
and $df_x(E^-_x) \subset E^-_{fx}$.

\bigskip \noindent
{\it From here on:} the setting in the Introduction is assumed. Let $l_a$ be the function in (H3).
We fix an $f$-invariant cutoff function $\la_\a$ with $l_\a < \la_\a < 0$, and apply the MET to the derivative cocycle $(f, \mu; df)$. All notation is as in the MET. Paring off sets of zero measure, we may assume that there exists an increasing sequence of Borel sets $K_1 \subset K_2 \subset \cdots \subset \As$ for which (i) $\oGamma = \cup_n K_n$ and (ii) the Oseledets subspaces $E_i, F$ are continuous on each $K_n$ (see Section 3.1 in \cite{BYentropy}).

As we will see, in most of our arguments it will suffice to consider one $\Gamma(\la^*; m, p)$
at a time. Specifically, from here to the end of Section 3, we fix $\la^*, m, p$, and write 
$$\Gamma = \Gamma(\la^*; m, p)\ .
$$
As we are interested only in splittings in which $df^n|E^-$ is strictly contracting, we may further assume $\la^* < 1/2p$.
}


\subsection{Adapted norms} \label{subsec:adaptedNormsAC}


We recall here without proof some standard results on adapted norms, modifying
results from Section 4 of \cite{BYentropy} as follows. 
Instead of decomposing the tangent space at $x \in \Gamma$
into $\Bc_x = E^u_x \oplus E^c_x \oplus E^s_x$ or $E^u_x \oplus E^{cs}_x$
as is done in \cite{BYentropy}, here we have $\Bc_x = E^+_x \oplus E^-_x$
where $E^+_x$ and $E^-_x$ are as defined in (\ref{defnEplus}) above;
we will sometimes refer to $E^-_x$ as the ``strong stable" direction.

Letting
\begin{align}\label{eq:laPlusMinusDef}
\lambda^+ = \lambda^* + \frac{1}{2 p} \qquad \mbox{and} \qquad  \lambda^- = \lambda^* - \frac{1}{2 p} \ ,
\end{align}
we have that $\lambda^- <0$, $\lambda^+ > \lambda^-$, and
$\lambda^+$ can be positive or negative.  Analogous to the construction in \cite{BYentropy}, we define the adapted norms $|\cdot|_x'$ as follows:
\begin{align*}
|u|'_x &= \sum_{n = - \infty}^{0} \frac{|df^{n}_x u |}{e^{n \la^+}} \text{ for } u \in E^{+}_x \, , \\
|w|'_x &= \sum_{n = 0}^{\infty} \frac{|df^n_xw |}{e^{n \la^-}} \text{ for } w \in E^-_x \, ,
\end{align*}
and for $v = u +  w \in \Bc_x, u \in E^{+}_x, w \in E^-_x$, we define $|v|_x' = \max \{|u|_x', |w|_x'\}$.

{ 
For $x \in \Gamma$ and $r > 0$, we will sometimes refer to the domain 
$\tilde B_x(r) = \{ v \in \Bc_x : |v|_x' \leq r\}$ equipped with the adapted norm $|\cdot|'_x$ 
as a ``chart", or a ``Lyapunov chart", a term borrowed from finite dimensional 
nonuniform hyperbolic theory. 
Accordingly, we define chart maps} $\tilde f_x : \tilde B_x(r) \to \Bc_{fx}$ by 
$\tilde f_x = \exp_{fx}^{-1} \circ f \circ \exp_x$. The proofs of the following results 
are identical to those in \cite{BYentropy}.

\begin{lem} \label{lem:charts1}\
\begin{itemize}
\vspace{-4pt}
\item[(a)] (One-step hyperbolicity) 
For any $u \in E^+_x, w \in E^-_x$, we have
\begin{gather*}
|df_x u|_{fx}'  \ge e^{{\lambda^+}} |u|'_x \\
|df_x w|_{f x}' \leq  e^{ \l^- } |w|'_x\ .
\end{gather*}
\end{itemize}
There exists $\d_1 > 0$ for which the following hold:
Given any $\d_2 > 0$, there
is a Borel measurable function $l:\Gamma \to \mathbb R^+$, with 
\begin{align}\label{eq:slowVaryAC}
l(f^\pm x) \le e^{\d_2}l(x)  \quad \text{ for } \mu-a.e. \, x \, ,
\end{align}
such that for all $x \in \Gamma$, 
\begin{itemize} \vspace{-4pt}
\item[(b)] the norms $|\cdot|'_x$ and $|\cdot|$ are related by
$$
\frac12 |v| \le |v|_x' \leq  l(x) |v|\ ;
$$
\item[(c)] for any $\d \le \d_1$, the following hold for 
$\tf_x$ restricted to $\tilde{B}_x(\d {l}(x)^{-1})$:
\begin{itemize}
\vspace{-3pt}
\item[(i)] $\lip'( \tf_x - (d \tf_x)_0) \leq \d$;
\item[(ii)] the mapping $z \mapsto (d\tilde f_x)_z$ satisfies 
$\lip' \big( d \tf_x \big) \leq {l}(x)$.
\end{itemize} \vspace{-4pt}
Here, $\Lip'$ refers to the Lipschitz constant taken with respect to the $|\cdot|'$ norm.
\end{itemize}
\end{lem}

\noindent Throughout, the parameters $\d_1$ and $\d_2 > 0$ are fixed
with $\d_2 \ll \la^+ - \la^-$, while $\d \le \d_1$ may be shrunk a finite number of 
times. 
The function $l : \Gamma \to [1,\infty)$ is as in Lemma \ref{lem:charts1}. Paring off a set of zero measure, we may
assume that \eqref{eq:slowVaryAC} holds for pointwise $x \in \Gamma$.

\smallskip

It follows from (a) and (b) above that for all $n \in \mathbb Z^+$,
\begin{align*}
|df^n_xw| & \le 2l(x) e^{n \l^-} |w| \qquad \mbox{for all} \quad w \in E^-_x\ , \\
|df^{-n}_xu| & \le 2l(x) e^{- n \l^+} |u| \qquad \mbox{for all} \quad u \in E^+_x\ .
\end{align*}
Hereafter, we write $\Gamma_{l_0} = \{x \in \Gamma : l(x) \leq l_0\}$ for $l_0 > 1$,
and refer to these as \emph{uniformity sets}.

\section{Preparation: $W^{ss}$-manifolds and transversals}\label{sec:transversals}

%

\noindent {\it Notation:} for $x \in \Gamma$, $r > 0$ we write $\tilde B^{\pm}_x (r) = \{v \in E^{\pm}_x : |v|'_x \leq r\}$, so that $\tilde B_x(r) = \tilde B_x^+(r) + \tilde B_x^-(r) $. We write $\pi^+_x, \pi^-_x$ for the projection operators corresponding to the splitting $\Bc_x = E^+_x \oplus E^-_x$, and for notational simplicity, we write $\tilde f^n_x$ instead of
$\tilde f_{f^{n-1}x} \circ \cdots \tilde f_{fx} \circ \tilde f_x$. We will sometimes omit mention of the point $x \in \Gamma$ at which the adapted norm is taken when it is clear from context.  

\subsection{Local strong stable manifolds}\label{subsec:localStrongStableMflds}

We state the following result, the proof
of which is identical to that of the usual local stable manifolds theorem; see, e.g., 
\cite{lianYoungHilbertMap}.

\begin{thm}\label{thm:stabManifold}
There is a constant $\d_1' \leq \d_1$ with the property that for all $\d \leq \d_1'$, 
there is a family of functions 
$\{h_x : \tilde B_x^-( \d l(x)^{-1}) \to \tilde B_x^{+}(\d l(x)^{-1})\}_{x \in \Gamma}$ such
that
\[
h_x(0) = 0 \qquad \text{ and } \qquad \tilde f_x (\graph h_x) \subset \graph h_{f x} \qquad \text{ for all } x \in \Gamma \, .
\]
With respect to the norms $|\cdot|_x'$, the family $\{h_x\}_{x \in \Gamma}$ has the following additional properties.
\begin{itemize}
\item[(a)] $h_x$ is $C^{1 + \Lip}$- Fr\'echet differentiable, with $(d h_x)_0 = 0$.
\item[(b)] $\Lip' h_x \leq \frac{1}{10}$ and $\Lip' d h_x \leq C l(x)$, where $C > 0$ is independent of $x$.  
\item[(c)] For any $z_1, z_2 \in \graph h_x$, we have the estimate
\[
|\tilde f_x z_1 - \tilde f_x z_2|_{f x}' \leq (e^{\la^-} + \d) |z_1 - z_2|_x' \, .
\]
\item[(d)] The set $\graph(h_x)$ is characterized by
\begin{eqnarray*}
\graph(h_x) & = & \{y \in \tilde B_x(\d l(x)^{-1}) : \tilde f^n_x y \in 
\tilde B_{f^nx}(\d (l(x)e^{\d_2 n})^{-1})  \\
& & \qquad \qquad \mbox{ and } |\tilde f^n_x y|_{f ^nx}' \leq (e^{\la^-} + \d)^n |y|_x'
\ \forall n \ge 1\}\ .
\end{eqnarray*}
\end{itemize}
\end{thm}

\noindent The {\it local strong stable manifold} at $x$, written 
$W_{{\rm loc}, x}^{ss}$, is defined to be $\exp_x \graph h_x$.

\medskip

Theorem \ref{thm:stabManifold} is obtained via the backwards graph transform, i.e., 
the graph transform taken with respect to $f^{-1}$. More precisely, we define 
$$
\Wc^{ss}_0(x) = \big\{ h : \tilde B^-_x (\d l(x)^{-1}) \to \tilde E^+_x :
h(0)=0 \mbox{ and } \Lip h \leq \frac{1}{10} \big\}\ .
$$
For $h \in \Wc^{ss}_0(x)$, we say the \emph{backwards graph transform} of $h$ is 
well defined and is equal to $\mathcal G_x h$
if there exists a unique mapping $\mathcal G_x h : \tilde B^-_{f^{-1} x}(\d l(f^{-1} x)^{-1}) \to E^+_{f^{-1} x}$ such that
\[
\tilde f_{ x} (\graph \mathcal G_x h) \subset \graph h \, .
\]
We state without proof the following  basic lemma which implies the existence and Lipschitzness of
the family $\{h_x\}$ above and which will be used again later on. 
 
\begin{lem}\label{lem:backwardsGT}
For all $\d \leq \d_1'$ and $x \in \Gamma$,  $\mathcal G_x : \Wc^{ss}_0(x) \to \Wc^{ss}_0(f^{-1} x)$ is well defined and is
a contraction mapping, i.e.  
\[
\| \mathcal G_x h_1 - \mathcal G_x h_2\|_{f^{-1} x, ss} \leq q \|h_1 - h_2 \|_{x, ss} 
\]
for $h_1, h_2 \in \Wc^{ss}_0(x)$, where the norm $\|\cdot\|_{x, ss}$ on $\Wc^{ss}_0(x)$ is defined by
\[
\|  h \|_{x, ss} =  \sup_{v \in \tilde B^-_x(\d l(x)^{-1})} \frac{|  h(v)|_x'}{|v|_x'} \, ,
\]
and $q \in (0,1)$ is a constant independent of $x$. 
\end{lem}

\subsection{Iterated transversals}\label{subsec:iteratedTrans}

In the local version of our result, the transversals $\Sigma^1$ and $\Sigma^2$
(see {\bf Theorem A}) will be pieces of manifolds contained inside the domains of a chart at some $x \in \Gamma$, and they will be of the form 
$\exp_x(\graph (g^i))$ for some $g^i: E_x^+ \to E_x^-, \ i=1,2$. 
For reasons to become clear momentarily,
 it will be necessary to consider shrinking charts. 
Let $\la_c := \frac12 \la^- < 0$ when $\la^+ > 0$, and $\la_c := \frac12(\la^+ + \la^-)$ when $\la^+ < 0$: the exponent $\la_c < 0$ will be the rate at which our charts shrink.

\begin{lem}\label{lem:forwardGT}
 For $\d > 0$ sufficiently small, the following holds. 
Fix $l_0 \geq 1$ and $c_0 \geq l_0$. Let $x \in \Gamma_{l_0}$, and let $g_0 : \tilde B^{+}_x(\d c_0^{-1}) \to \tilde B^-_x(\d c_0^{-1})$ be a $C^{1 + \Lip}$-Fr\'echet differentiable map for which $\Lip' (g_0) \leq \frac{1}{10}$.
Then, writing $c_n = e^{-n \la_c} c_0$, there exists a sequence of $C^{1 + \Lip}$ maps 
$$
g_n : \tilde B^{+}_{f^n x}(\d c_n^{-1}) \to \tilde B^-_{f^n x}(\d c_n^{-1})\ , \quad n \geq 1\ , 
$$
with the following properties:
\begin{itemize}
\item[(a)] $\{g_n\}_{n \geq 0}$ is a (forward) graph transform sequence along the charts $\{\tilde B_{f^n x}(\d c_n^{-1}) \}_{n \geq 0}$, i.e., for all $n \geq 1$ we have 
$$
\graph g_{n + 1} = \tilde B_{f^{n + 1} x} (\d c_{n + 1}^{-1}) \cap  \tilde f_{f^{n} x} (\graph g_n ) \, ,
$$
\item[(b)] For all $n \geq 1$, we have that:
  \begin{itemize}
  \item $\Lip' (g_n) \leq \frac{1}{10}$, and   
  \item $\Lip' (d g_n) \leq  2 e^{n \d_2} (l_0 + \Lip' (d g_0))$. 
  \end{itemize}
\item[(c)] Let $0 \leq k < n$ and let $u_n^i \in \tilde B^{+}_{f^n x}(\d c_n^{-1})$, $i = 1,2$. Then $u_n^i +  g_n(u_n^i) = \tilde f_{f^k x}^{n-k} (u_k^i + g_k(u_k^i))$ for some $u_k^i \in \tilde B^{+}_{f^k x}(\d c_k^{-1})$, and
\[
|u_n^1 - u_n^2|_{f^n x}' \geq (e^{\la^+} - \d)^{n-k} |u_k^1 - u_k^2|_{f^k x}' \, .
\]
\end{itemize}
\end{lem}

As $E^+$ is finite-dimensional, the proof of Lemma \ref{lem:forwardGT} follows from standard graph transform arguments which we summarize here: 
\begin{itemize}
\item[(1)] Even though $d \tilde f_{ x}|_{E^+_{ x}}$ may be contracting, the presumptive domain $\pi^+_{f x} \circ \tilde f_{ x} \big( \graph g_0 \big)$ of the graph transform $g_{ 1}$ contains $\tilde B_{f x}^+(\d c_{1}^{-1})$ because $\la_c < \la^+$; the same 
comment applies to all subsequent steps. 
\item[(2)] Since $\graph g_0$ need not pass through $0$, it must be checked that each 
$\graph g_n$ sits inside the diminished
chart at $f^n x$; this is ensured because
there is a point $z \in \graph g_0 \cap \graph h_{ x}$, where $W_{{\rm loc}, x}^{ss} = \exp_x \graph h_x$, and $\tilde f^n_x(z)$ tends to $0$ much more quickly than the rate at which chart sizes shrink, i.e., $\la_c > \la^-$.
\end{itemize}

Though not yet justified at this point, we will refer to the manifolds $\exp_{f^nx}(\graph(g_n))$ where the $g_n$ are as in Lemma  \ref{lem:forwardGT} as 
{\it transversals} to the $W^{ss}$-foliation.

We record below two properties of the sequence of transversals defined by $\{g_n\}$. 
The first says that they become increasingly ``flat" in a sense to be made precise, and the second gives a distortion estimate on 
the dynamics restricted to these transversals.
The setting and notation are as in
Lemma \ref{lem:forwardGT}.

\begin{lem}\label{lem:expFlatness}
The sequence of functions $g_n$ has the property that
\begin{equation}
\label{eq:gnFlat}
\sup_{u \in \tilde B^+_{f^n x} (\d c_n^{-1}) } |(d g_n)_u|_{f^n x}' \quad \to \ 0
\quad \mbox{ exponentially fast as } n \to \infty\ ,
\end{equation}
with uniform bounds (independent of $x$) depending only on $l_0$, $\la^+, \la^-$
and $\Lip'(dg_0)$.
\end{lem}

\begin{proof} Let $z_0 \in \graph(g_0) \cap \graph (h_x)$, $z_n = \tilde f^n_x(z_0)$,
and write $z_n = u_n + g_n(u_n)$. From $z_n \in \graph h_x$, it follows from contraction along $W^{ss}$-leaves and a standard graph transform argument that 
\begin{align}\label{eq:flatGraph}
|(dg_n)_{u_n}|' \lesssim e^{n (\la^- - \la^+ + \d)} \, ,
\end{align} where $\lesssim$ refers to inequality up to a multiplicative constant depending only on $l_0$. The lemma follows from this, together with the fact that $\Lip' (d g_n) \leq  2 e^{n \d_2} (l_0 + \Lip' (d g_0))$ and the domain
of $g_n$ has diameter $\d c_0^{-1} e^{\la_c n}$, which shrinks faster than $\Lip' (d g_n)$
can grow. 
\end{proof}

\smallskip

\begin{lem}\label{lem:distTransversals}
For any $l, L \geq 1$ there exists a constant $D_{l, L} > 0$ with the following property. Let $x \in \Gamma$, $c_0 \geq l(x)$, and let $g_0$ be as in Lemma \ref{lem:forwardGT}; set $L_0 = \Lip'(d g_0)$. Then, for any $n \geq 1$ and $y^1, y^2 \in f^{-n} (\exp_{f^n x} \graph g_n)$, we have the estimate
\[
\bigg| \log \frac{\det ( df^n_{ y^1} | T_{ y^1} (\exp_x \graph g_0)) }{\det ( df^n_{ y^2} | T_{ y^2} (\exp_x \graph g_0)) } \bigg| \leq D_{l(x), L_0}   \cdot \max\big\{ \big(e^{- \d_2}(e^{\la^+} - \d) \big)^{-n} , 1 \big\} \cdot |f^n y^1 - f^n y^2| \, ,
\]
where $T_yW$ denotes the tangent space to the manifold $W$ at $y$.
\end{lem}

\begin{proof}
The case $\la^+ > 0$ follows verbatim from the proof of Proposition 5.8 in \cite{BYentropy}. In the case $\la^+ < 0$, Lemma \ref{lem:distTransversals} follows from similar arguments to those in \cite{BYentropy}. The only substantive difference is that an expansion estimate along unstable manifolds is replaced with the following `weak contraction' estimate along transversals (c.f. Lemma \ref{lem:forwardGT}, item (c)):
\[
|z^1_k - z^2_k|_{f^k x}' \leq (e^{\la^+} - \d)^{- (n - k)} |z^1_n - z^2_n|_{f^n x}' 
\]
for $0 \leq k < n$; here we have written $z^i_k = \exp_{f^k x}^{-1} f^k y^i$ for $i = 1,2$ and $0 \leq k \leq n$.

Another difference is that the constant $D$ appearing in the distortion estimate depends now on the Lipschitz constant $L_0$ of $d g_0$; this, however, does not substantially change the arguments in \cite{BYentropy}.
\end{proof}

\subsection{Continuity of holonomy maps along the $W^{ss}$ ``foliation"}\label{sec:stronglyStable}

In preparation for the proof of absolute continuity of holonomy maps,  we first establish
their continuity, which we carry out in some detail, following the outline below:

\medskip
Step 1. continuity of $x \mapsto E^-_x$ \ (Lemma \ref{lem:contSplitting})

Step 2. continuity of $k$-step backward graph transforms \ (Lemma \ref{lem:contGraphTform})

Step 3. continuity of $x \mapsto W^{ss}_{{\rm loc}, x}$ \ (Lemma \ref{lem:stableStacks})\ , and finally

Step 4. continuity of holonomy maps along $W^{ss}_{\rm loc}$-leaves \ (Lemma \ref{lem:continuousHolonomy})

\medskip

We begin with the continuity of the distribution $E^-$.
For $l_0 > 1$ we write 
\[
\Gamma^+_{l_0} = \{x \in \Gamma : |d f^n_x|_{E^-_x}| \leq l_0 e^{n \la^-}  \text{ for all } n \geq 0\} \, ;
\]
the sets $\Gamma^+_{l_0}$ are referred to as \emph{forward uniformity sets}, as they only detect information along \emph{forward} trajectories. 

\begin{lem}\label{lem:contSplitting}
Let $l_0 > 1$ be fixed. Then, $x \mapsto E^-_x$ varies continuously in the Hausdorff metric $d_H$ as $x$ varies in $\Gamma_{l_0}^+$.
\end{lem}

Observe by Lemma \ref{lem:charts1} that $\Gamma_{ l_0} \subset \Gamma^+_{2 l_0}$ for any $l_0 > 1$, and so Lemma \ref{lem:contSplitting} implies the continuity of $x \mapsto E^-_x$ across the uniformity sets $\Gamma_{l_0}$ as well.

\begin{proof}[Proof of Lemma \ref{lem:contSplitting}]
Let $x^n \to x$ be a convergent sequence in $\Gamma_{l_0}^+$. To show $d_H(E^-_{x^n}, E^-_x) \to 0$, we will prove $\Gap(E^-_{x^n}, E^-_x) \to 0$ as $n \to \infty$, where $\Gap$ is as defined in Sect. \ref{subsec:bSpaceGeometryAC} (see Lemma \ref{lem:symmetryCloseness}). 
Assume the contrary. Then there exists a sequence of unit vectors $v^n \in E^-_{x^n}$ such that, writing
\[
v^n = w^{n, +} + w^{n, -}
\]
according to the splitting $\Bc = E^+_x \oplus E^-_x$, we have $|w^{n, +}| \geq c $ for some constant $c > 0$. 

We use the shorthand $x_k = f^k x, x_k^n = f^k x^n$. Then for arbitrary $k, n$,
\begin{align*}
(*) := |df^k_{x} v^n| & \geq |df^k_x w^{n, +}| - |df^k_x w^{n, -}| \\
& \ge \frac12 l(x)^{-1} e^{k(\la^+ - \d_2)} c - l_0 e^{k \la^-} \|\pi^-_x\|\ ,
%
\end{align*}
after carrying out the change of norms and using $x \in \Gamma^+_{l_0}$. On the other hand, since $x^n \in \Gamma^+_{l_0}$ and $v_n \in E^-_{ x^n}$,
\[
(*) = |df^k_x v^n| \leq |(df^k_x - df^k_{x^n}) v^n| + |df^k_{x^n} v^n| \leq \|df^k_x - df^k_{x^n}\| + l_0 e^{k \la^-}  \, .
\]
Taking the limit as $n \to \infty$, we have shown that for all $k$,
\[
(1 + \|\pi^-_x\|) l_0 e^{k \la^-} \geq \frac12 c l(x)^{-1} e^{k (\la^+-\d_2)} \, .
\]
For $k$ large enough, this is a contradiction. 
\end{proof}

Next we treat the continuity of backward graph transforms. Let $\mathcal G_x$ be as defined in Sect. \ref{subsec:localStrongStableMflds}.
It is easy to see that this transform can be extended to the set of functions
\begin{eqnarray*}
\mathcal W^{ss}_{\frac{1}{10}}(x) & := & \{h \in \tilde B^-_x(\d l(x)^{-1}) \to \tilde E^+_x  : 
(i) \Lip(h) < \frac{1}{10}\  \mbox{ and } \\
& & \qquad (ii) \ \exists \hat z, z \mbox{ with } |\hat z|'_{f^{-1}x}, |z|'_x < \frac{1}{10} \d l(x)^{-1} \mbox{ s.t. } \tilde f_{f^{-1}x} \hat z = z \in \graph h\}\ .
\end{eqnarray*}
The following notation will be used for `chart switching': Let $x, y \in \Gamma$ and let $\phi_y : \Dom(\phi_y) \to E^{+}_y$ be a Lipschitz map, where $\Dom(\phi_y) \subset E^-_y$. We write $\phi_y^{x} : \Dom(\phi_y^{x}) \to E^{+}_x$ for the map, if it exists, for which
\[
\exp_y \graph \phi_y = \exp_x \graph \phi_y^x \, .
\]

\begin{lem}\label{lem:contGraphTform} Let $x, y^n \in \Gamma_{l_0}$ be such that $y^n \to x$ as $n \to \infty$,
and fix arbitrary $k \in \mathbb Z^+$. Writing $x_k=f^k x$ and $y^n_k =f^ky^n$,
we let ${\bf 0}_{x_k}: E^-_{x_k} \to E^+_{x_k}$ and 
${\bf 0}_{y^n_k} : E^-_{y^n_k} \to E^+_{y^n_k}$ be the functions that are identically 
equal to zero. Then for all large enough $n$,
${\bf 0}^{x_k}_{y^n_k}: \tilde B^-_{x_k}(\d (l_0 e^{k \d_2})^{-1}) \to E^+_{x_k}$ are defined, as are the backward graph transforms
$$
\mathcal G^k_{x_k}({\bf 0}^{x_k}_{y^n_k}):=
\mathcal G_{x_1} \circ \dots \circ \mathcal G_{x_{k-1}} \circ \mathcal G_{x_k}
({\bf 0}^{x_k}_{y^n_k})\ .
$$
as mappings $\tilde B^-_{x}(\d l_0^{-1}) \to \tilde B^+_x(\d l_0^{-1})$; moreover, $\|\mathcal G^k_{x_k}({\bf 0}^{x_k}_{y^n_k}) - 
\mathcal G^k_{x_k}({\bf 0}_{x_k})\|_{C^0}' \to 0$ as $n \to \infty$, where $\|\cdot\|_{C^0}'$ is taken with respect to the adapted norm $|\cdot|'_x$.
\end{lem}

\begin{proof} As $E^-$ is continuous on $\Gamma_{l_0e^{k\d_2}}$, the well-definedness of
${\bf 0}^{x_k}_{y^n_k}$ and $\mathcal G^k_{x_k} {\bf 0}^{x_k}_{y^n_k}$ for large enough $n$ is clear, and the only statement that requires a proof is 
the last statement on
$C^0$ convergence. For this it suffices to prove the continuity of the backward graph
transform for one step. For definiteness, let us work with $\tilde \Gamma_{x_1}$.
We will show that for $h_1, h_2 \in 
\mathcal W^{ss}_{\frac{1}{10}}(x_1)$, we have $\|\mathcal G_{x_1}(h_1)
-\mathcal G_{x_1}(h_2)\|_{C^0}' \le C \|h_1-h_2\|_{C^0}'$ for $C = 2 e^{-\la^+}$.

Fix arbitrary $v \in \tilde B^-_x(\d l_0^{-1})$, and let 
$z_i = (\tilde \Gamma_{x_1}h_i)(v)+v, i=1,2$. We omit the subscripts in $\tilde f_x$,
$\pi^\pm_x$, $|\cdot|'_x$ etc. when they are obvious from context.
By Lemma \ref{lem:forwardGT}(b),(c), 
\begin{equation} \label{Gamma}
|\pi^-(\tilde f z_1-\tilde f z_2)|' \le \frac{1}{10} |\pi^+(\tilde f z_1-\tilde f z_2)|'
\qquad \mbox{and} \qquad 
|\pi^+(\tilde f z_1-\tilde f z_2)|' \ge (e^{\lambda^+} - \d) |z_1-z_2|'\ .
\end{equation}
Let $w = \pi^- \tilde f z_1$. Then 
\begin{eqnarray*}
|h_2(w)-h_1(w)|' & \ge & |\pi^+(\tilde f z_2 - \tilde f z_1)|' - |h_2(w)-\pi^+(\tilde f z_2)|'\\
& \ge & |\pi^+(\tilde f z_2 - \tilde f z_1)|' - 
\frac{1}{10} |\pi^-(\tilde f z_1-\tilde f z_2)|' \qquad \mbox{since} \ \Lip(h_2)<\frac{1}{10}\\
& \ge & \frac{99}{100} |\pi^+(\tilde f z_2 - \tilde f z_1)|'  \qquad \mbox{ by } (\ref{Gamma})\ .
\end{eqnarray*}
Using (\ref{Gamma}) again, we conclude from this that
\[
|z_1-z_2|' \le 2 e^{-\lambda^+} |h_2(w)-h_1(w)|' \le 2 e^{-\lambda^+} \|h_1-h_2\|_{C^0}\ . \qedhere
\]
\end{proof}

The next lemma defines what we will refer to as {\it a stack of strong stable leaves}. 
Below, we write $B^{\pm}_x(r) = \{v \in E^{\pm}_x : |v| \leq r\}$ (note the difference between $\tilde B$ and $B$). 

\begin{lem}\label{lem:stableStacks}
Let $l_0$ and $n_0 > 1$ be fixed, and fix $x_0 \in \Gamma_{l_0} \cap K_{n_0}$ ($K_n$ as in the end of Section 2.2). For $\e>0$ and $x \in \Gamma$, we let 
\[U(x, \e) := \Gamma_{l_0} \cap K_{n_0} \cap \{y:|x-y|<\e\}\ ,\] 
and let $\{h_x\}_{x \in \Gamma}$ be as in Theorem \ref{thm:stabManifold}. Then, for any $\d \leq \frac14 \d_1'$, there exists $\e_0 > 0$ sufficiently small so that the following hold:
\begin{itemize}
\item[(a)] For any $y \in U(x_0, \e_0)$, the map $h_y^{x_0}$ is defined on $B^-_{x_0}(\d l_0^{-3})$ with $\Lip (h_y^{x_0}|_{B^-_{x_0}(\d l_0^{-3})}) \leq \frac{1}{10}$. 
\item[(b)] The mapping $\Theta : U(x_0, \e_0) \to C^0(B^-_{x_0}(\d l_0^{-3}) , B^+_{x_0}(\d l_0^{-3}))$ defined by setting $\Theta(y) = h_y^{x_0}|_{B^-_{x_0}(\d l_0^{-3})}$ is continuous in the uniform norm.
\end{itemize}
\end{lem}

\begin{proof} (a) follows from the continuity of $x \mapsto E^+_x, E^-_x$ on $K_{n_0}$
and Theorem \ref{thm:stabManifold}; the extra copies of $l_0^{-1}$
come from norm changes and the reduction of domain size to keep the graphs ``flat".
For more detail see the proof of Lemma 5.5 in \cite{BYentropy}. 

For (b), we fix $x, y^n \in U(x_0, \e_0)$ with $y^n \to x$ as $n \to \infty$. To prove
the continuity of $\Theta$ at $x$, it suffices to
show that given any $\gamma>0$, when restricted to $B^u_x(2\d l_0^{-3})$
we have $\|h^x_{y^n} - h^x_x\| < \gamma$ for all large enough $n$ (here $\|\cdot\|$ refers to the uniform ($C^0$) norm taken using the standard norm $|\cdot|$ on $\Bc$).
For $k \in \mathbb Z^+$, write $x_k=f^kz$ and $y^n_k=f^ky_n$. Since $x_k, y^n_k \in 
\Gamma_{l_0 e^{k\d_2}}$, we have, by Lemma \ref{lem:contSplitting}, $E^-_{y^n_k} \to E^-_{x_k}$
as $n \to \infty$. (We could not have concluded this from the continuity of $E^-$ on 
$K_{n_0}$ alone because we may have $df_x(E^-_x) \subsetneq E^-_{fx}$.)
Using the notation in Lemma \ref{lem:contGraphTform}, we have that
$$
\|h^x_{y^n} - h^x_x\| \le \|h^x_{y^n} -\mathcal G^k_{x_k}({\bf 0}^{x_k}_{y^n_k})\|
+ \|\mathcal G^k_{x_k}({\bf 0}^{x_k}_{y^n_k}) - \mathcal G^k_{x_k}({\bf 0}_{x_k})\|
+ \|\mathcal G^k_{x_k}({\bf 0}_{x_k}) - h^x_x\|\ .
$$
From Lemma \ref{lem:backwardsGT} and the uniform equivalence of $|\cdot|$ and $|\cdot|'$ norms on
uniformity sets, we have that the first and third terms above are $<\gamma/3$ for $k$
sufficiently large. Fixing one such $k$, Lemma \ref{lem:contGraphTform} tells us that the middle
term is $< \gamma/3$ for $n$ large enough, completing the estimate.
\end{proof}

Letting $\bar U \subset U(x_0, \e_0)$ be any compact subset, we refer to 
\begin{align}\label{eq:ssStackDefn}
\Sc = \bigcup_{x \in \bar U} \exp_{x_0} \graph \Theta(x)
\end{align}
as a \emph{stack of strong stable leaves}. We remark that for $x,y \in \bar U$,
either $\Theta(x) = \Theta(y)$ or $\Theta (x) \cap \Theta (y) = \emptyset$. This
follows easily from Theorem \ref{thm:stabManifold}(d).

We finish with a lemma on the continuity of holonomy maps.

\begin{lem}\label{lem:continuousHolonomy} Let $\Sc = \bigcup_{x \in \bar U} \exp_{x_0} \graph \Theta(x)$ be 
as above. For $i=1,2$, we let 
\[
\Sigma^i = \exp_{x_0} \graph \sigma^i \, , \qquad \mbox{where} \qquad
\sigma^i : B^+_{x_0}(\d l_0^{-3}) 
\to B^-_{x_0}(\d l_0^{-3})
\]
has $ \Lip(\sigma^i) \le \frac{1}{10}$, and let 
 \[
\check \Sigma^i = \Sc \cap \Sigma^i\ .
\]
Then the holonomy map $p:\check \Sigma^1 \to \check \Sigma^2$ defined by
letting $p(z)$ be the unique point in \linebreak
$\exp_{x_0} (\graph \Theta (x)) \cap \Sigma^2$
for $z \in \exp_{x_0}(\graph \Theta(x)) \cap \Sigma^1$ is a homeomorphism. 
\end{lem}

\begin{proof} Define $\psi_1: \bar U \to \check \Sigma^1$ by $\{\psi_1(x)\}=
\exp_{x_0} \graph \Theta(x) \cap \Sigma^1$. First we observe that $\psi_1$ is continuous by Lemma \ref{lem:stableStacks}(b).
In more detail, if $x, y^n \in \bar U$ and $y^n \to x$,
$$
|\psi_1(y^n)-\psi_1(x)| \le |\pi_{x_0}^+(\psi_1(y^n) -\psi_1(x))| + 
|\pi_{x_0}^-(\psi_1(y^n) - \psi_1(x))| := A^+ + A^-\ , 
$$ 
with $A^- \le \frac{1}{10} A^+$ by the condition on $\Lip(\sigma^1)$.
Letting $w = \pi_{x_0}^-\psi_1(x)$, we also have, by the condition 
on $\Lip(\Theta(y^n))$ in Lemma \ref{lem:stableStacks}(a),
$$
|\Theta(y^n)w - \Theta(x)w| \ge A^+ - \frac{1}{10} A^-\ .
$$
Thus $|\Theta(y^n)w - \Theta(x)w| \to 0$ as $n \to \infty$ implies
$|\psi_1(y^n)-\psi_1(x)| \to 0$.
Introducing the equivalence relation $\sim$ on $\bar U$ where
$x \sim y$ iff $\Theta(x)=\Theta(y)$, $\psi_1$ gives rise to a continuous map
$\check \psi_1: (\bar U/\sim) \to \check \Sigma^1$, which is injective and therefore
 a homeomorphism by the compactness of $\bar U$. Defining $\check \psi_2: (\bar U/\sim) \to \check \Sigma^2$ analogously,
we have that $p=\check \psi_2 \circ  \check \psi_1^{-1}$ is a homeomorphism.
\end{proof}

\section{Local version of absolute continuity result}\label{sec:localAbsCty}

In Sect. \ref{subsec:preciseFormulationAC}, we formulate the precise statement of {\bf Theorem A} and give 
an outline of the proof. Details are given in Sects. \ref{subsec:holonomiesOmega} and \ref{subsec:cover}.

\subsection{Theorem A: precise formulation and outline of proof}\label{subsec:preciseFormulationAC}

{\bf Setting.} {
Let $(f, \mu)$ be as in Section 1, satisfying (H1)--(H3). We fix 
$\Gamma = \Gamma(\la^*; m, p)$ with $\la^*< 1/2p$ and $\mu(\Gamma)>0$.}
The setup consists of a stack $\Sc$ of strong stable manifolds and
a pair of transversals $\Sigma^1$ and $\Sigma^2$ to the leaves of $\Sc$. 
More precisely, we apply the constructions of Section \ref{subsec:adaptedNormsAC}. 
Fix $l_0 > 1$ 
and $n_0 \in \N$. Let $x_0 \in \Gamma_{l_0} \cap K_{n_0}$ ($K_n$ as in the end of Section 2.2). We fix
$\d \leq \frac14 \d_1'$ small enough for the results in Section 3 to apply. Let $\e_0 > 0$ is as in Lemma \ref{lem:stableStacks},
and let $\Sc$ be the stack of strong stable leaves defined as in (\ref{eq:ssStackDefn}) 
through points in a compact set $\bar U \subset U(x_0, \e_0) \subset 
\Gamma_{l_0} \cap K_{n_0}$. For the transversals,
we let $\Sigma^i = \exp_{x_0} (\graph \sigma_i)$, $i=1,2$, where $\sigma_i : B^{+}_{x_0}(2 \d l_0^{-3}) \to B^-_{x_0}(\frac12 \d l_0^{-4})$ are $C^{1 +\Lip}$ maps satisfying $\Lip (\sigma_i) \leq \frac{1}{40 l_0}$.
These conditions ensure that for all $x \in \bar U$, 
$g_0^i := \sigma_i^x|_{\tilde B^+_x(\d l_0^{-3})}$ satisfies the assumptions in 
Lemma \ref{lem:forwardGT} { (see Lemma \ref{lem:switchChart})}.
As in Lemma \ref{lem:continuousHolonomy}, we define $\check \Sigma^i := \Sigma^i \cap  \Sc$, and let $p: \check \Sigma^1 \to \check \Sigma^2$ be the holonomy map.

\bigskip \noindent
{\bf Theorem A} {\it In the setting above, assume that $\nu_{\Sigma^1}(\check \Sigma^1) > 0$. Then the holonomy map $p$ is absolutely continuous with respect to the induced volumes $\nu_{\Sigma^1}$ and $\nu_{\Sigma^2}$ restricted to $\check \Sigma^1$ and $ \check \Sigma^2$, respectively.
Moreoever, $p$ has uniformly bounded Jacobian, i.e., there exists a constant $C > 0$ with the property that for any Borel set $A \subset \check \Sigma^1$,
\[
C^{-1} \nu_{\Sigma^1}(A) \leq \nu_{\Sigma^2}\big(p(A)\big) \leq C \nu_{\Sigma^1}(A) \, .
\] 
}

The goal of Section 4 is to prove this result. Section 5 proves an explicit 
formula for the Radon-Nikodym derivative of $p$.

\begin{rmk} {\bf Theorem A} has been proved a number of times for 
diffeomorphisms of finite dimensional Riemannian manifolds \cite{pughShub, pesin, katokStrelcyn}. 
We remark on some of the more significant issues in passing from 
these settings to Banach space maps. The first is that  {\it a priori} 
there is no notion of volume on transversals in Banach spaces. 
There is, however, a
well defined measure class, namely that generated by Haar measure on
finite dimensional subspaces, and that is adequate for the definition of
absolute continuity for $W^{ss}$-foliations, but not for statements
 on Radon-Nikodym derivatives
of holonomy maps. A second issue is that in the proofs, one needs to compare
Jacobians of high iterates of the map, at different phase points and restricted
to different subspaces. This requires not only the introduction of volume elements on finite dimensional subspaces (which we have done in \cite{BYentropy}) but 
proofs of regularity of volume elements and determinants as subspaces
are varied; Proposition \ref{prop:detReg} is in this spirit. Finally,
as we will see, absolute continuity of the $W^{ss}$-foliation ultimately boils
down to one's ability to pass ``round balls", or sets with nice geometries,
between nearby, roughly parallel transversals (see, e.g., \cite{pughShub, katokStrelcyn, youngDiffSurv}). Banach-space geometry is not always nice; 
indeed in some Banach spaces, $x \mapsto |x|$ is not even differentiable. 
Finite dimensional techniques such as overcovering by round balls with controlled intersections have no obvious analogs in Banach spaces.
\end{rmk}

We have found that it is technically simpler to work with the following surrogate for
balls on embedded submanifolds.

\begin{defn}\label{defn:Bballs}
Let $W \subset \Bc$ be an embedded submanifold.  For $x \in W$ and $r > 0$ we define the $\Omega$-ball of radius $r$ in $W$ centered at $x$ to be
\[
\Omega_{W}(x, r) = \text{ the connected component of } W \cap \{y \in \Bc : |x - y| \leq r\} \text{ containing } x \, .
\]
\end{defn}

We are primarily interested in the case where $W$ is a finite dimensional embedded $C^1$ submanifold and $r > 0$ is very small.

\bigskip \noindent
{\bf Notation:} Below and throughout Section 4 we use the shorthand $\Sigma^i_n = f^n \Sigma^i, \  \check \Sigma^i_n = f^n \check \Sigma^i$, and write $p_n : \check \Sigma^1_n \to \check \Sigma^2_n$ for the conjugated holonomy $p_n := f^n \circ p \circ f^{-n}$. The symbols $\lesssim, \gtrsim$ denote $\leq, \geq$, respectively, up to a multiplicative constant independent of $n$ (but perhaps depending on $l_0$); the symbol $\approx$ means that both of $\lesssim$ and $\gtrsim$ hold.

\bigskip \noindent
{\bf Outline of proof:} It suffices to show  there exists a constant $C > 0$ such
that for every compact set $A \subset \check \Sigma^1$, we have
 $\nu_{\Sigma^2}(p(A)) \leq C \nu_{\Sigma^1}(A)$. This is because all
 bounded Borel sets can be approximated from the inside by compact sets,
and the other inequality can be obtained by reversing the roles of $\Sigma^1$
and $\Sigma^2$.
 Let $A$ be given, and let $\Oc \supset A$ be an open neighborhood for which $\nu_{\Sigma^1}(\Oc ) \leq 2 \nu_{\Sigma^1}(A)$. We will show that for some large $n$, there is a collection of open $\Omega$-balls $\{\Omega_1, \Omega_2, \cdots, \Omega_M\}$ of $\Sigma^1_n$ for which $\{f^{-n}\Omega_i\}$ has the following properties:
\begin{itemize}
\item[(a)] $A \subset \bigcup_i f^{-n} \Omega_i  \subset \Oc$;
\item[(b)] $\sum_i \nu_{\Sigma^1} \big( f^{-n} \Omega_i \big) \lesssim  \, \nu_{\Sigma^1} \big( \bigcup_i f^{-n} \Omega_i \big)$;
\item[(c)] $\nu_{\Sigma^2}(p(\check \Sigma^1 \cap f^{-n} \Omega_i)) \lesssim \nu_{\Sigma^1} \big( f^{-n} \Omega_i \big)$ for each $i$.
\end{itemize}
 From (a)--(c), it follows immediately that 
\begin{eqnarray*}
\nu_{\Sigma^2}(p(A)) & \le & \sum_i \nu_{\Sigma^2}(p(\check \Sigma^1 \cap f^{-n} \Omega_i))\\
& \lesssim  & \sum_i \nu_{\Sigma^1}( f^{-n} \Omega_i) \qquad
\mbox{by (c)}\\
& \lesssim & \nu_{\Sigma^1} \big( \bigcup_i f^{-n} \Omega_i \big) \qquad
\mbox{by (b)}\\
& \lesssim & \nu_{\Sigma^1}(\Oc) \le 2 \nu_{\Sigma^1}(A) \qquad \mbox{by (a) and
the choice of } \Oc\ ,
\end{eqnarray*}
giving the desired result.

To complete the proof, then, 
it suffices to produce $\{\Omega_1, \cdots, \Omega_M\}$ with
properties (a)--(c) above, and to be sure that the constants in ``$\lesssim$" are
independent of $A$. In the proof to follow, $\Omega_i$  will be chosen to be $\Omega$-balls (in the sense of Definition \ref{defn:Bballs}), and they will be of the form $\Omega_{\Sigma^1_n}(f^n y_i, e^{n\la_b})$ for suitable choices of $y_i \in \check \Sigma^1$. Here 
$\la_b<0$ is a new lengthscale satisfying  
$$\la^- < \la_b < \la_c < \la^+\ .
$$
We assume $\la_b$ is fixed and bounded away from $\la^-$ and $\la_c$ by small 
numbers to be specified. 

\subsection{Holonomies of ``large" $\Omega$-balls}\label{subsec:holonomiesOmega}

To prove {\bf Theorem A}, we need to show $\nu_{\Sigma^2}(p(A)) 
\approx \nu_{\Sigma^1}(A)$ for all Borel subsets $A \subset \check \Sigma^1$. 
We consider in this subsection a situation where $A$ is an $\Omega$-ball 
 the radius of which is much larger than the distance between the
 two transversals, and explain how that is relevant to the original problem.

\begin{lem}\label{lem:switchChart}
The following hold with uniform bounds for all $x \in \bar U$. For $i = 1,2$ let $g_0^i = \sigma_i^x|_{\tilde B^+_x(\d l_0^{-3})}$ (using the
notation just before Lemma \ref{lem:contGraphTform}). Then, $g_0^i$ has range contained in $\tilde B^-_x(\d l_0^{-3})$, and we have the estimates $\Lip'(g_0^i) \leq 1/10$ and $\Lip'(dg_0^i) \leq 5 l_0 \Lip(d \sigma^i)$. Here, $\Lip'$ refers to the adapted norm $|\cdot|_x'$ at $x$.
\end{lem}

The proof of Lemma \ref{lem:switchChart} follows from the considerations in Section 5.2 in \cite{BYentropy}, which we do not repeat here. Lemma \ref{lem:switchChart} permits us to apply forward graph transforms as in Lemma \ref{lem:forwardGT}, with $c_0 = l_0^3, c_k = e^{-k \la_c} c_0$, to obtain the graph transform sequence $\{g_k^i : \tilde B^+_{f^k x}(\d c_k^{-1}) \to \tilde B^-_{f^k x}(\d c_k^{-1})\}_{k \geq 0}$ in the charts system along the trajectory $\{f^n x\}_{n \geq 0}$ for any $x \in \bar U$. 

In what follows we will use the notation $\Sigma^i_{n,x}:= \exp_{f^n x} \graph g^i_n$, and will be comparing $\Omega$-balls in $\Sigma^i_{n,x}, i=1,2$ for some
large $n$. In addition to other quantities to be specified, it should be assumed throughout
that the choice of $n$ will depend implicitly on $\d, l_0$ 
and $\Lip(d \sigma^i)$ (but it must not depend on $x \in \bar U$).

\medskip
In the rest of Sect. \ref{subsec:holonomiesOmega}, we fix attention on an arbitrary $y^1 \in \check \Sigma^1$. We write $p(y^1) = y^2$ and $y^i_n = f^n y^i$, and let $x \in \bar U$ be the corresponding point for which $y^i \in W^{ss}_{{\rm loc}, x}$ for $i = 1,2$. We first
establish that sets of the form $\Omega_{\Sigma^i_n}(y^i_n, 10e^{n \la_b})$ lie well 
inside the domains of the charts system along $\{f^n x\}$.

\begin{lem}\label{lem:chartContainment}
For all $n$ large enough, we have

(i) $
\Omega_{\Sigma^i_n} (y^i_n, 10 e^{n \la_b}) \subset \Sigma^i_{n,x} \, 
$

(ii) $\diam (f^{-n} \Omega_{\Sigma^1_n}(y^1_n, e^{n \la_b})) \to 0$ as $n \to \infty$\ .
\end{lem}

\begin{proof}
(i) Let $\hat y^i \in \Omega_{\Sigma^i_n}(y^i_n, 10 e^{n \la_b})$. We estimate:
\begin{align*}
|\hat y^i - f^n x|_{f^n x}' &\leq |\hat y^i - y^i_n|_{f^n x}' + |y^i_n - f^n x|_{f^n x}' \leq l(f^n x) |\hat y^i - y^i_n| + (e^{n \la^-} + \d)^n |y^i - x|_x' \\
& \leq 10 l_0 e^{n (\d_2 + \la_b)} + (e^{n \la^-} + \d)^n \cdot \d c_0^{-1} \, ,
\end{align*}
which is $\leq \frac12 \d c_n^{-1}$ assuming $e^{\la^-} + \d < e^{\la_b}$ and 
$\d_2 + \la_b < \la_c$. 

As for (ii), notice that for $k=1,2,\dots, n$, $f^{-k}\Omega_{\Sigma^1_n}
(y^1_n, 10e^{n \la_b}) \subset \Sigma^i_{n-k,x}$, 
so that \\
$\diam(f^{-n}\Omega_{\Sigma^1_n}(y^1_n, 10e^{n \la_b})) \sim e^{n \la_b} (e^{-\la^+}+\d)^n$,
which tends to $0$ as $n \to \infty$.
\end{proof}

\begin{lem}\label{lem:conjHolonomyContainment}
For any $\e > 0$, there exists $n$ sufficiently large (depending on $\e$)
for which we have
\[
p_n \big( \check \Sigma^1_n \cap \Omega_{\Sigma^1_n}(y^1_n, e^{n \la_b}) \big) \subset \Omega_{\Sigma^2_n}(y^2_n, (1+\epsilon) e^{n \la_b})
\]
\end{lem}

\begin{proof}
Let $\hat y_n \in \check \Sigma^1_n$, and let $\hat x \in \bar U$ be such that $f^{-n} \hat y_n \in W^{ss}_{{\rm loc}, \hat x}$. Then
\begin{align}\label{eq:pnNearIdentity}
|\hat y_n - p_n (\hat y_n)| \leq 2 |\hat y_n - p_n(\hat y_n)|_{f^n \hat x}' \leq 2 (e^{\la^-} + \d)^n \cdot \d c_0^{-1} \, .
\end{align}
If, additionally, $\hat y_n \in \Omega_{\Sigma^1_n}(y^1_n, e^{n \la_b})$, then
\[
|p_n(\hat y_n) - y^2_n| \leq |p_n(\hat y_n) - \hat y_n| + |\hat y_n - y^1_n| + |y^1_n - y^2_n| \leq e^{n \la_b} + 4 \d c_0^{-1} (e^{\la^-} + \d)^n \leq (1 + \e) e^{n\la_b}
\]
proving the containment. 
\end{proof}

We consider next the relation between $\Omega$-balls in $\Sigma^i_{n,x}$ 
and in $E^+_{f^nx}$. Define $\Psi^i = \Psi^i_{n,x} : \Sigma^i_{n,x} \to 
E^+_{f^nx}$ by $\Psi^i:= \pi^+_{f^n x} \circ \exp_{f^n x}^{-1}$.

\begin{lem}\label{cor:expFlat} Fix $\e > 0$. Then for all $n$ sufficiently large (depending on 
$\e$), we have
\begin{itemize}
\item[(i)] $\frac{1}{1 + \e}  |y - y'| \leq |\Psi^i(y) - \Psi^i(y')| \leq (1 + \e) |y - y'|$ \, for all $y, y' \in \Sigma^i_{n,x}$,
\item[(ii)] 
$\frac{1}{1 + \e} \leq  \det (d \Psi^i) \leq 1 + \e \, .
$
\end{itemize}
\end{lem}

\begin{proof} For (i), observe that $y = f^n x + \Psi^i(y) + g_n^i \circ \Psi^i(y)$ and similarly for $y'$, so that
\[
|y - y'| \leq (1 + \Lip g_n^i) | \Psi^i(y) - \Psi^i(y')| \leq (1 + \e) | \Psi^i(y) - \Psi^i(y')| 
\]
on taking $n$ sufficiently large so that $|(d g_n^i)_u| \leq \e$ for $u \in \tilde B^+_{f^n x}(\d c_n^{-1})$ by Lemma \ref{lem:expFlatness} (note that we have passed from the adapted norm $|\cdot|'_{f^n x}$ in the conclusions of Lemma \ref{lem:expFlatness} to the standard norm $|\cdot|$); the lower bound works similarly.

For (ii), note that $\Psi^i = (\exp_{f^n x} \circ (\Id + g_n^i))^{-1}$. The desired estimate follows on applying Lemma \ref{lem:expFlatness} to make $|dg^i_n|$ sufficiently small and on applying the simple bound $(1 - |V|)^{\dim E} \leq \det(\Id + V | E) \leq (1 + |V|)^{\dim E}$ to $V = dg^i_n, E = E^+_{f^n x}$.
\end{proof}

\medskip
We summarize the results thus far {\it vis a vis} the Outline of proof 
in Sect. \ref{subsec:preciseFormulationAC}:
For given small $\e$ and $n$ large enough depending on $l_0, \Lip(d \sigma^i)$ and
$\e$, we have
shown that $\Omega^1 :=
\Omega_{\Sigma^1_n} (y^1_n, e^{n \la_b})$ has the properties
\begin{itemize}
\item[(i)] $f^{-n}(\Omega^1) \subset \Oc$ \qquad Lemma \ref{lem:switchChart}(ii)
\item[(ii)] $p_n(\Omega^1 \cap \check \Sigma_n^1) \subset \Omega_{\Sigma^2_n} 
(y^2_n, (1+\e) e^{n \la_b}):= \Omega^2$ \qquad Lemma \ref{lem:chartContainment}
\item[(iii)] $\Psi_n^1(\Omega^1) \supset \Omega_{E^+_{f^nx}}(\Psi^1_n y^1_n, (1+\e)^{-1} e^{n\la_b})$ 

\qquad and \ \ $\Psi_n^2(\Omega^2) \subset \Omega_{E^+_{f^nx}}(\Psi^2_n y^2_n, (1+\e)^2 e^{n\la_b})$ \qquad Lemma \ref{lem:conjHolonomyContainment}(i)
\end{itemize}
Notice that $E^+_{f^nx}$ is a linear subspace, and $\Omega$-balls in 
$E^+_{f^nx}$ are usual Banach space balls.
Combining the above and using Lemma \ref{lem:conjHolonomyContainment}(ii), we obtain
\begin{eqnarray*}
\nu_{\Sigma^2_n}(\Omega^2) & \le & (1+\e) \cdot m_{E^+_{f^nx}} \Omega_{E^+_{f^nx}}(\Psi^2_n y^2_n, (1+\e)^2 e^{n\la_b}) \\
& = & (1+\e) \cdot  (1+\e)^{3 \dim E^+_{f^nx}} \cdot
m_{E^+_{f^nx}} \Omega_{E^+_{f^nx}}(\Psi^1_n y^1_n, (1+\e)^{-1} e^{n\la_b})\\
& \le & (1+\e)^{2 + 3 \dim E^+_{f^nx}} \cdot \nu_{\Sigma^1_n}(\Omega^1)
\ .
\end{eqnarray*}
Here we have used the translation invariance and scaling properties of
the induced volumes $m_{E^+_{f^nx}}$ {on the linear subspaces} 
$E^+_{f^nx}$ (see Sect. 2.1). 
The discussion above suggests that we take $\{\Omega_i\}$ in the Outline to 
consist of sets of the form $\Omega^1$.

\subsection{A cover by $\Omega$-balls}
\label{subsec:cover}

We construct here the cover $\{\Omega_1, \cdots, \Omega_M\}$ of $f^n(A)$ in the Outline in Sect. \ref{subsec:preciseFormulationAC}. We continue to use the notation from Sect. \ref{subsec:holonomiesOmega}, but as we will be working exclusively with iterates of $\Sigma^1$, we will drop the superscript 
$1$ in $\Sigma_n^1$. We say a cover has {\it multiplicity}
$\le C$ if no point is contained in more than $C$ elements of the cover.

\begin{prop}\label{prop:BCL}
Let $A \subset \check \Sigma$ be compact, and let 
$n$ be large enough that Lemma \ref{lem:conjHolonomyContainment} holds with $\e=1$. 
Then there is a finite set $S = S_{n, A} = \{y_i\}_{i =1 }^M \subset f^n A$ with the property that $\{\Omega_{\Sigma_n}(y_{i}, e^{n \la_b})\}_{i = 1}^M$ is a cover of $f^n A$ with multiplicity $\leq C_{m^+}$, where the constant $C_{m^+}$ depends only on $m^+ := \dim E^+$.
\end{prop}

\begin{proof} Writing $r=e^{n \la_b}$, we  take $S= \{y_i\}_{i =1 }^M$ to be a  $\big(\frac{r}{2} \big)$-\emph{maximal separated set} in $f^n A$, i.e., 
\begin{itemize}
\item[(a)] $\Omega_{\Sigma_n}(y_{i}, \frac{r}{2}) \cap \Omega_{\Sigma_n}(y_{j}, \frac{r}{2}) = \emptyset$ for any $1 \leq i < j \leq M$, and 
\item[(b)] for any 
$y \in f^n A$, $\Omega_{\Sigma_n}(y, \frac{r}{2}) \cap \Omega_{\Sigma_n}(y_{i}, \frac{r}{2}) \neq \emptyset$ for some $i \in \{1, \cdots, M\}$.
\end{itemize}
That such a set exists and is finite follows from the compactness of $f^n A$ and of 
$\Sigma_n$ for all $n > 0$; details are left to the reader.

To complete the proof, we will show that (i) $\{\Omega_{\Sigma_n} (y_{i}, r) \}_{i = 1}^M$ is a cover of $f^n A$ and (ii) the multiplicity of this cover is
bounded by a constant depending only on $m^+$. 

That (i) holds follows from the following: Given $y \in f^nA$, let $y_{i}$ be
given by property (b), and let $z \in \Omega_{\Sigma_n}(y, \frac{r}{2}) \cap \Omega_{\Sigma_n}(y_{i}, \frac{r}{2})$. Since $y$ and $z$ both lie in a
connected component of $\Sigma_n \cap \{w \in \Bc: |w-y| \le \frac{r}{2}\}$, 
there is a continuous path in $\Sigma_n \cap \{w: |w-y| \le \frac{r}{2}\}$
connecting $y$ and $z$. Likewise, there is a continuous path in
$\Sigma_n \cap \{|w-y_{i}| \le \frac{r}{2}\}$ connecting $z$ and $y_{i}$.
Concatenating these two paths, we obtain that $y \in 
\Omega_{\Sigma_n}(y_{i},r)$.

To prove (ii), for each $i$ we let $S_i =  \{j \in \{1, \cdots, M\} \setminus \{i \} : \Omega_{\Sigma_n}(y_{i}, r) \cap \Omega_{\Sigma_n}(y_{j}, r) \neq \emptyset\}$.
Then the multiplicity of the cover $\{\Omega_{\Sigma_n} (y_{i},r) \}_{i = 1}^M$ is no worse than 
\[
\max_{1 \leq i \leq M} \# S_i + 1 \, ,
\] so it suffices to bound $\# S_i$ by a constant depending only on $m^+ = \dim E^+$. 

For fixed $i$, we let $x_i \in f^n\bar U$ be such that $f^{-n} y_{i} \in 
W^{ss}_{{\rm loc},f^{-n}x_i}$, and recall that $\Omega_{\Sigma_n}(y_{i}, 10 r) \subset \tilde B_{x_{i}}(\d c_n^{-1})$ (Lemma \ref{lem:chartContainment}). Letting 
$\Psi = \pi^+_{x_{i}} \circ \exp_{x_{ i}}^{-1}$, we now pass from
$\Omega$-balls in $\Sigma_n$ to balls in $E^+_{x_i}$ via  Corollary \ref{cor:expFlat}:
For $j \in S_i$, since $\Omega_{\Sigma_n}(y_{j}, \frac{r}{2}) \subset \Omega_{\Sigma_n}(y_{i}, 3 r)$, we have 
$\Omega_{E^{+}_{x_{i}}}(\Psi y_{j}, \frac{r}{4}) \subset
\Omega_{E^{+}_{x_{i}}}(\Psi y_{i}, 6r)$. As the sets
$\Omega_{\Sigma_n}(y_{j}, \frac{r}{2})$ are pairwise disjoint by property
(a) above, so are the sets $\Omega_{E^{+}_{x_{i}}}(\Psi y_{j}, \frac{r}{4})$. 
By volume count, the maximum number of such sets that can fit inside
$\Omega_{E^{+}_{x_{i}}}(\Psi y_{i}, 6r)$
is no more than $24^{\dim E^+}$. We have thus shown that
 $\# S_i \leq 24^{\dim E^+}$, completing the proof.
\end{proof}

Summarizing what we have proved {\it vis a vis} the Outline in Sect. \ref{subsec:preciseFormulationAC}: 
In addition to the requirement in Proposition \ref{prop:BCL}, let $n$ be chosen large 
enough that $f^{-n}\Omega_{\Sigma_n} (f^nz, e^{n \la_b}) \subset \Oc$ for all $z \in A$, 
and let $\{\Omega_i \}$ be the cover
$\{\Omega_{\Sigma_n} (y_{i}, e^{n \la_b} ) \}_{i = 1}^M$ in Proposition \ref{prop:BCL}.
Then (a) and (b) in the Outline hold, the constant in ``$\lesssim$" in (b)
being the multiplicity of this cover. 

\subsection{Completing the proof}

We continue to use the notation in Sects. \ref{subsec:preciseFormulationAC} and \ref{subsec:holonomiesOmega}. To prove the  remaining item in the Outline, item (c), it suffices to prove the following. 

\begin{prop}\label{prop:massEstimate}
There exists $D>0$
such that for all $n$ sufficiently large and for all $y \in \check \Sigma^1$, 
\[
\nu_{\Sigma^2} \bigg( p \big(\check \Sigma^1 \cap f^{-n} \Omega_{\Sigma^1_n}(f^ny, e^{n \la_b}) \big) \bigg) \leq D \nu_{\Sigma^1} \big( f^{-n} \Omega_{\Sigma_n^1}(f^ny, e^{ n \la_b}) \big) \, .
\]
\end{prop}

\begin{proof}
By Lemma \ref{lem:conjHolonomyContainment}, it suffices to bound from
above the ratio
\[
(*) := \frac{\nu_{\Sigma^2} \big(  f^{-n} \Omega_{\Sigma^2_n}(f^np(y), (1 + \e) e^{n \la_b}) \big) }{\nu_{\Sigma^1} \big( f^{-n} \Omega_{\Sigma^1_n}(f^ny,  e^{n \la_b}) \big)}\ .
\]
By the change of variables formula, 
$\nu_{\Sigma^1} \big( f^{-n} \Omega_{\Sigma^1_n}(f^ny,  e^{n \la_b})\big)$
is related to $\nu_{\Sigma^1_n}(\Omega_{\Sigma^1_n}(f^ny,  e^{n \la_b}))$ by
the Jacobian of $f^n|(f^{-n} \Omega_{\Sigma^1_n}(f^ny,  e^{n \la_b}))$, and
this in turn is related to the corresponding Jacobian at the point $y$
by the distortion estimate in Lemma \ref{lem:distTransversals}. 
Examining the estimate in Lemma \ref{lem:distTransversals},
note that the $\Omega$-balls we consider have radius $e^{n \la_b}$, which 
contracts faster than the contraction rate $e^{n \la^+}$ along transversals. Thus 
for any fixed $\e > 0$, we may take $n$ sufficiently large (depending on 
$\e$, $l_0$ and the Lipschitz constant $L_0 = \Lip'(d g_0^i)$) so that the
 the right-hand side in Lemma \ref{lem:distTransversals}
 is $\leq \log(1 + \e)$.

Applying this estimate to both the numerator and denominator of $(*)$ and invoking Lemma \ref{lem:conjHolonomyContainment}, we obtain 
\[
(*) \leq (1+\e)^2  \cdot \underbrace{\frac{\det(df^n_{y} | T_{y} \Sigma^1) }{\det(df^n_{p(y)} | T_{p(y)} \Sigma^2)}}_{I} \cdot \underbrace{\frac{\nu_{\Sigma^2_n} \Omega_{\Sigma^2_n} (f^np(y), (1 + \e) e^{n \la_b})}{\nu_{\Sigma^1_n} \Omega_{\Sigma^1_n} (f^ny, e^{n \la_b})}}_{II}\ .
\]
As Term $II$ has been bounded at the end of Sect. \ref{subsec:holonomiesOmega}, it remains to bound
Term $I$. For $N \in \N$, we introduce the function $\Delta_N : \check \Sigma^1 \to [0,\infty)$ by
\[
\Delta_N(y) := \frac{\det(df_{y}^N | T_{ y} \Sigma^1)}{\det(df_{  p (y)}^N | T_{ p (y)} \Sigma^2)} 
=  \prod_{n = 0}^{N-1} \frac{\det(df_{f^n y} | T_{f^n y} \Sigma^1_n)}{\det(df_{f^n p (y)} | T_{f^n p (y)} \Sigma^2_n)} \, .
\]
To complete the proof of Proposition \ref{prop:massEstimate}, it suffices to show that there exists $\check D$ (depending only on 
$l_0$) such that $ \check D^{-1} \leq \Delta_N(y) \leq \check D$ for all $y \in \check \Sigma^1$
and for all $N \in \mathbb Z^+$. This follows from Lemma \ref{lem:infProduct} below. 
\end{proof}

We prove a stronger result than needed here, namely the existence of the 
$N \to \infty$ limit, which is needed 
in Section 5. Observe that $y \mapsto \Delta_N(y)$ is continuous in $y \in \check \Sigma^1$ for any fixed $N > 0$ by the continuity of $p$ (Lemma \ref{lem:continuousHolonomy}) and the regularity of $\det$ (Proposition \ref{prop:detReg}). 

\begin{lem}\label{lem:infProduct}
For any $y \in \check \Sigma^1$, the limit $\Delta(y):= \lim_{N \to \infty} \Delta_N(y)$ exists. The convergence $\Delta_N \to \Delta$ is uniform, so $y \mapsto \Delta(y)$ is therefore continuous. Moreover, there is a constant $ \check D =  \check D_{l_0} > 0$ (depending only on $l_0$) such that $ \check D^{-1} \leq \Delta(y) \leq \check D$ for any $y \in  \check \Sigma^1$.
\end{lem}

\begin{proof}[Proof of Lemma \ref{lem:infProduct}] Given $N$ and $y$, we define for $k \in \mathbb Z^+$
\[
\Delta_{N,k}(y) = \frac{\det(df^{k}_{f^N y} | T_{f^N y} \Sigma^1_N)}{\det(df^{k}_{f^N p(y)} | T_{f^N p( y)} \Sigma^2_N)} =
  \prod_{n = 0}^{k-1} \frac{\det(df_{f^{N+n} y} | T_{f^{N+n} y} \Sigma^1_{N+n})}{\det(df_{f^{N+n} p (y)} | T_{f^{N+n} p (y)} \Sigma^2_{N+n})} \, .
\]
We will show there exists $D_1$ (depending on $l_0$ but
not on $y$ or $N$) such that for all $k \ge 1$:
\begin{align}\label{eq:uniformM}
\Delta_{N,k}(y) \le D_1 \cdot e^{N \big( \frac12 (\la^- - \la^+) + 2 \d_2 \big)} \, .
\end{align}
The proof of (\ref{eq:uniformM}) relies on regularity properties of the determinant function, which should not be taken granted as our notion of volume on finite dimensional subspaces
was defined one subspace at a time.
We state formally the estimate used:

\begin{cla}\label{cla:distcontrol}
Let $\hat x \in \Gamma$, $\hat y^1, \hat y^2 \in 
\tilde B_{\hat x}(\d l(\hat x)^{-1})$, and let $L_1, L_2 : E^+_{\hat x} \to E^-_{\hat x}$ be linear maps for which $|L_i|_{\hat x}' \leq 1/10$ for $i =1,2$. Write $E_i = (\Id + L_i) E^+_{\hat x}$. Then, we have the estimate
\begin{align}\label{eq:stabDistortion}
\bigg| \log \frac{\det (df_{\hat y^1}|E_1)}{\det(df_{\hat y^2} |E_2) } \bigg| \leq \text{Const. }  l(\hat x)^q \big( |\hat y^1 - \hat y^2| + |L_1 - L_2| \big) \, ,
\end{align}
where $q \in \N$ depends only on $\dim E^+$.
\end{cla}

To deduce this inequality from Proposition \ref{prop:detReg}, observe that $M$ as in
Proposition \ref{prop:detReg} is determined by $|(df_{\hat y^i}|_{E_i})^{-1}| \leq 2 e^{- \la^+} l(\hat x)$, and
$\epsilon$ can be taken as small as need be by introducing intermediate points of the 
form $\hat y^{1,\ell}:=\hat y^1+\ell \gamma(\hat y^2-\hat y^1)$ and linear maps $
L_{1,\ell}:= L_1 + \ell \gamma (L_2 - L_1)$ for $\ell = 1,2, \dots \gamma^{-1}, \gamma \ll 1$, and applying
Proposition \ref{prop:detReg} to $\hat y^{1,\ell}$ and $\hat y^{1,\ell+1}$, $L_{1,\ell}$
and $L_{1,\ell+1}$. The constant on the right side of \eqref{eq:stabDistortion} 
follows from Remark \ref{rmk:constantsDependence} after Proposition \ref{prop:detReg}.

\smallskip
Letting $x \in \bar U$ be such that $y \in W^{ss}_{{\rm loc},x}$ and using Claim \ref{cla:distcontrol} we estimate
\begin{align} \label{eq:partialProduct}\begin{split}
\bigg| & \log \frac{\det(df^k_{f^N y} | T_{f^N y} \Sigma^1_N)}{\det(df^k_{f^N p(y)} | T_{f^N p( y)} \Sigma^2_N)} \bigg| 
 \leq \sum_{n = 0}^{k-1} \bigg| \log \frac{\det(df_{f^{N + n} y} | T_{f^{N + n} y} \Sigma^1_{N + n})}{\det(df_{f^{N + n} p(y)} | T_{f^{N + n} p( y)} \Sigma^2_{N + n})}  \bigg| \\
& \leq \text{Const. } \sum_{n = 0}^{k-1} l(f^{N + n} x)^q \cdot \big( |f^{N + n} y - f^{N + n} p (y)| + |( d g_{N + n}^1)_{u^1_{N + n}} - (dg_{N + n}^2)_{u^2_{N + n}}| \big) \, ,
\end{split}\end{align}
where $u^1_n = \pi^+_{f^{N+n} x} \circ \exp_{f^{N+n} x}^{-1}( f^{N+n} y)$ and
$u^2_n = \pi^+_{f^{N+n} x} \circ \exp_{f^{N+n} x}^{-1}(f^{N+n} p(y))$.
To bound the RHS of \eqref{eq:partialProduct}, recall the estimates $l(f^{N+n} x) \leq e^{(N+n) \d_2} l_0$ and $|f^{N+n} y - f^{N+n} p(y) | \lesssim (e^{ \la^-} + \d)^{N+n}$. For the last term, the estimate \eqref{eq:flatGraph} in the proof of Lemma \ref{lem:expFlatness} gives
 the bound
\[
 |( d g_{N+n}^1)_{u^1_{N+n}} - (dg_{N+n}^2)_{u^2_{N+n}}| \lesssim e^{(N+n) \big(\frac12 (\la^- - \la^+) + \d_2 \big)} \, .
\]
Assuming, as we may, that $\d_2 \ll \frac14( \la^- - \la^+)$, the desired
result follows.
\end{proof}

The proof of {\bf Theorem A} is now complete.


\section{Derivative computation}\label{sec:derivativeComputation}

The setting is as in the beginning of Section 4. We now compute explicitly the Radon-Nikodym derivative 
of the holonomy map $p$. 

\medskip \noindent 
{\bf Theorem B} {\it For all $y \in \check \Sigma^1$,
\[ 
\frac{d (p^{-1}_*\nu_{\Sigma^2})}{d \nu_{\Sigma^1}}(y) = \Delta(y) \, 
\]
where $\Delta(y)$ is given by Lemma \ref{lem:infProduct}.
}

\subsection{Outline of proof}\label{subsec:outlineDerAC}

As $ y \mapsto \Delta(y)$ is continuous on $\check \Sigma^1$, by considering small
enough sets on which $\Delta(y)$ is nearly constant, one deduces {\bf Theorem B} 
from

\begin{prop}\label{prop:main2}
Let $A \subset \check \Sigma^1$ be compact. Then,
\[
\nu_{\Sigma^2}\big(p(A) \big) \leq \sup_{y \in A} \Delta(y) \cdot \nu_{\Sigma^1}(A) \, .
\]
\end{prop}

Below, we fix $\e > 0$, to be regarded as acceptable error in our pursuit of the inequality in Proposition \ref{prop:main2}. 
As the main source of the overestimate in the proof of {\bf Theorem A} comes from the overcovering by $\Omega$-balls of $f^n A$, we now replace this over-cover
by a collection of pairwise disjoint sets. An important requirement for this
new cover is that
the volumes of its elements must be transformed  nicely by holonomy maps, 
a property we have, up until now, proved only for $\Omega$-balls 
that are large in radius 
compared to the distance between transversals (Sect. \ref{subsec:holonomiesOmega}).

\medskip
\noindent {\it Construction of a special cover.} For $n$ sufficiently large, we apply Proposition \ref{prop:BCL} to obtain a cover $\{\Omega_i\}_{i = 1}^M$ of $f^n A$ by balls of the form $\Omega_i = \Omega_{\Sigma^1_n}(y_i, e^{n \la_b})$, where $\{y_i\} \subset f^n A$. 
Writing $\a \Omega_i = \Omega_{\Sigma_n^i}(y_{ i}, \a e^{n \la_b})$ for $\a > 0$, we
define the collection $\{V_i\}_{i =1 }^M$ of pairwise disjoint measurable sets that
will comprise this special cover as follows:
\[
V_{i} = \Omega_i \setminus \bigg( \bigcup_{j < i } \Omega_j \ \cup \ \bigcup_{i < j } \frac12 \Omega_j  \bigg) \, .
\]
The following are immediate:
\begin{itemize}
\item[(i)] $\frac12 \Omega_i \subset V_i \subset \Omega_i$ for any $1 \leq i \leq M$, and
\item[(ii)] $ f^{-n} \big( \cup_i V_i \big)  \supset A$.
\end{itemize}
Let $S_i$ be as in the proof of Proposition \ref{prop:BCL}, that is to say, $S_i$ consists of those
indices $j \ne i$ such that $\Omega_j \cap \Omega_i \ne \emptyset$. Observe that only those $\Omega_j$ 
with $j \in S_i$ are involved in the construction of $V_i$, and that as shown in Proposition \ref{prop:BCL}, the cardinality of $S_i$ is bounded by a constant that depends only on $m^+ = \dim E^+$.
In partciular, it is independent of $M$, which can grow exponentially with $n$. 
As we will see, our control
on the ``geometry" of the sets $V_i$ will depend crucially on this uniform bound on the
cardinality of $S_i$.

\medskip

Continuing to allow dependence on $\d, l_0$ and $\Lip(\sigma^i)$, our main estimate is the following:

\begin{lem}\label{lem:partitionPassesWell}
Assume that $n$ is sufficiently large depending on $\e > 0$. Then, for any $1 \leq i \leq M$ we have that
\[
\nu_{\Sigma^2_n} \big( p_n(V_i \cap \check \Sigma^1_n) \big) \leq (1 + \e) \nu_{\Sigma^1_n}(V_i) \, .
\]
\end{lem}
The proof of Lemma \ref{lem:partitionPassesWell} is deferred to the next subsection.

\begin{proof}[Proof of Proposition \ref{prop:main2} assuming Lemma \ref{lem:partitionPassesWell}]  Let $\e>0$ be given.
We fix an open set $\Oc \supset A$ with the property that $\nu_{\Sigma^1}(\Oc \setminus A) \leq \e \, \nu_{\Sigma^1}(A)$. The value of $n$ will be increased a finite number of times
as we go along. First we assume it is large enough that
$f^{-n}V_i \subset \Oc$ where $\{V_i\}$ is as constructed above. We then
bound $\nu_{\Sigma^2}\big( p(A) \big)$ by
\begin{align}\label{eq:partSplit}
\nu_{\Sigma^2}\big( p(A) \big) \leq \sum_{i = 1}^M \nu_{\Sigma^2} \big( f^{-n} p_n(V_i \cap \check \Sigma^1_n) \big) \, .
\end{align}
As before, we have, from Lemma \ref{lem:distTransversals},
\[
\frac{\nu_{\Sigma^2} \big( f^{-n} p_n(V_i \cap \check \Sigma^1_n) \big)}{\nu_{\Sigma^1}(f^{-n} V_i)}
\leq
(1 + \e)^2 \Delta_n(f^{-n} y_{i}) \cdot \frac{\nu_{\Sigma^2_n} \big( p_n(V_i \cap \check \Sigma^1_n) \big) }{\nu_{\Sigma^1_n} (V_i)} \, ,
\]
and assume $n$ is large enough that $\Delta_n \le (1+\e) \Delta$ on $\check 
\Sigma^1$ (Lemma \ref{lem:infProduct}). Applying these inequalities together with Lemma \ref{lem:partitionPassesWell} to the right side of \eqref{eq:partSplit} and
summing, we obtain
\[
\nu_{\Sigma^2}\big( p(A) \big) \leq (1 + \e)^3 \sup_{y \in A} \Delta(y) \cdot \nu_{\Sigma^1}(\Oc) \leq (1 + \e)^4 \sup_{y \in A} \Delta(y) \cdot \nu_{\Sigma^1}(A) \, .
\]
Taking $\e \to 0$ completes the proof.
\end{proof}

\subsection{Proof of Lemma \ref{lem:partitionPassesWell}}\label{subsec:proofLemmaPartPassesWell}

For fixed $n$ and $i$, we let $V^1_i=V_i$ be as defined in the last subsection, and 
extend this notation  in the following ways: 

\medskip \noindent
(i) Let $y_j^1 = y_j, y_j^2 = p_n(y^1_j)$, and define $V^2_i$ analogously, 
with $y^2_j$ in the place of $y^1_j$.

\medskip \noindent
(ii) For $k=1,2$ and $\a>1$, we define
\begin{align}\label{eq:fattenPartitionAC}
V_i^k(\a) := \a \Omega_i^k \setminus \bigg( \bigcup_{\substack{j < i }} \a^{-1} \Omega_j^k \ \cup \ \bigcup_{\substack{j > i }} \frac{ \a^{-1}}{2} \Omega_j^k  \bigg) \, .
\end{align}
Notice that $V^k_i(\a) \subset V^k_i(\a')$ for $\a < \a'$, and for $\a>1$, 
the sets $V^k_i(\a)$ and $V^k_j(\a)$ are not necessarily pairwise disjoint. 

\medskip \noindent
(iii) We consider next analogous constructions on $E^+$.
Let $\Psi_i = \pi^+_{x_i} \circ (\exp_{x_i})^{-1}$ be projection to $E^+_{x_i}$ where
$x_i$ is such that $f^{-n}x_i \in \bar U$ and $y^k_i \in W^{ss}_{{\rm loc},x_i}$.
Let $S'_i:=\{j : y^1_j \in \Omega_{\Sigma^1_n}(y^1_i, 3e^{\la_b n})\}$. For $j \in S'_i$, we
let $\bar \Omega_{j}^k = \Omega_{E^+_{x_i}}(\Psi_i (y_j^k), e^{n \la_b})$, and define, for 
$\a \approx 1$, 
\[
\bar V_i^k(\a) := \a \bar \Omega_{i}^k \setminus \bigg( \bigcup_{\substack{j < i, j \in S'_i}} \a^{-1} \bar \Omega_{j}^k \ \cup \ \bigcup_{\substack{j > i, j \in S'_i}} \frac{ \a^{-1}}{2} \bar \Omega_{ j}^k  \bigg) \, .
\]
Reasoning similar to those in Section 4 shows that for $n$ large enough, 
$j \in S'_i$ are the only indices involved in the definition of $\bar V_i^k(\a)$, and
that $\# S_i' \leq C_{m^+}'$ for all $i$, where $C_{m^+}'$ depends on $m^+ = \dim E^+$ alone.
It is important to note that $\bar V_i^k(\a)$ is {\it not} the $\Psi_i$-image of
$V^k_i$, and that it is constructed using real balls in $E^+_{x_i}$, the centers of which
are projections of those used in the construction of $V^k_i$.

\begin{proof}[Proof of Lemma \ref{lem:partitionPassesWell}] Let $\e>0$ be given. We first choose 
$\e'=\e'(\e)>0$ and $\a=\a(\e, \e')>1$ with $\e', |\a-1|$
sufficiently small, and then $n=n(\e, \e', \a)$ sufficiently large; exact dependences will become clear in the course of the proof. Let
$\{V^1_i\}$ be a special cover of $f^n(A)$ as defined in Sect. \ref{subsec:outlineDerAC}. 
We assume $\{V^1_i\}$
 is constructed from $\{\Omega^1_i\}$, where each $\Omega^1_i = 
\Omega_{\Sigma^1_n}(y^1_i, e^{\la_b n})$, and
let $i$ be fixed throughout. We will show that the assertion in
Lemma \ref{lem:partitionPassesWell} follows from the following sequence of approximations:
\begin{itemize}
\item[(1)] $p_n(V^1_i \cap \check \Sigma^1_n) \subset V^2_i(\a)$,
\item[(2)]
$\Psi_i(V^1_i) \supset \bar V^1_i(\a^{-1})$ \ and \ $\Psi_i (V^2_i(\a)) \subset
\bar V^2_i(\a^2)$ ,
\item[(3)] $\bar V^2_i(\a^2) \subset \bar V^1_i(\a^3)$ ,
\item[(4)] $m_{E^+_{x_i}}(\bar V^1_i(\a^3)) \le (1+\e') m_{E^+_{x_i}}(\bar V^1_i(\a^{-1}))$ .
\end{itemize}
Applying (1)-(4) in the order stated together with Lemma \ref{cor:expFlat}(ii), we
obtain
\begin{eqnarray*}
\nu_{\Sigma^2_n}(p_n(V_i \cap \check \Sigma^1_n)) & \le & 
\nu_{\Sigma^2_n}(V^2_i(\a))\\
& \le & (1+\e') \ m_{E^+_{x_i}}(\bar V^2_i(\a^2)) \\
& \le & (1+\e')\  m_{E^+_{x_i}}(\bar V^1_i(\a^3)) \\
& \le & (1+\e')^2 \ m_{E^+_{x_i}}(\bar V^1_i(\a^{-1}))\\
& \le & (1+\e')^3 \ \nu_{\Sigma^1_n}(V^1_i)\ .
\end{eqnarray*}

It remains to prove (1)-(4).

\medskip \noindent
{\it Proof of (1).} Lemma \ref{lem:conjHolonomyContainment} asserts that for $n$ large enough,
$$p_n(\Omega_i^1 \cap \check \Sigma^1_n) 
\subset \a \Omega_i^2\ .
$$
A similar proof applied
to $p_n^{-1}$ gives, for $j \in S_i$, 
$$p_n^{-1}(\a^{-1} \Omega_j^2 \cap \check \Sigma^2_n) 
\subset \Omega_j^1 \qquad \mbox{and} \qquad
p_n^{-1}(\frac12 \a^{-1} \Omega_j^2 \cap \check \Sigma^2_n) \subset
\frac12 \Omega^1_j.
$$
Combining these relations give the desired result.

\medskip \noindent
{\it Proof of (2).} This follows from the bi-Lipschitz property of $\Psi_i$ wth Lipschitz constant $\approx 1$ (Lemma \ref{cor:expFlat}(i)). 
It implies in particular
$\Psi_i(\Omega^1_i) \supset \a^{-1} \bar \Omega^1_i$ and $\Psi_i(\Omega^1_j) \subset \a \bar \Omega^1_j$
for $j \in S_i$, the latter being valid because $\Omega^1_j \subset 
\Omega_{\Sigma^1_n}(y_i^1, 10 e^{\la_b n})$. The second containment is proved similarly.

\medskip \noindent
{\it Proof of (3).} It suffices to estimate $|\Psi_i(y^1_j) -\Psi_i(y^2_j)|$ 
where $j \in S'_i \cup \{i\}$; the rest of the containments are as before. This quantity is equal to
$$
|\pi^+_{x_i}(p_n(y^1_j)- y^1_j)| \le |\pi^+_{x_i}| \cdot |p_n(y^1_j)- y^1_j|
\le 2 l_0 e^{n\d_2} \cdot 2 \d c_0^{-1}(e^{\la_-} + \d)^n\ ,
$$
which can be made arbitrarily small relative to $e^{\la_b n}$ by taking $n$ large.

\medskip \noindent
{\it Proof of (4).} We will show
\begin{equation} \label{vol}
m_{E^+_{x_i}}(\bar V^1_i(\a^3)) \setminus \bar V^1_i(\a^{-1}))
\le \e' m_{E^+_{x_i}}(\bar V^1_i(\a^{-1}))\ .
\end{equation}
Let $\omega$ denote the volume of the unit ball in $\mathbb R^{m^+}$
where $m^+ = \dim E^+$. Then the left side of \eqref{vol} is bounded from above by
\[
e^{\la_b n m^+} \omega [(\a^{3m^+} -\a^{-m^+}) + (\#S'_i)(\a^{m^+} - \a^{-3m^+})]\ .
\]
As for the right side of \eqref{vol}, recall that we have made sure $\bar V^1_i(\a')
\supset \frac12 \bar \Omega^1_i$ for any $\a' > 1, \a' - 1 \ll 1$. Thus
$$
m_{E^+_{x_i}}(\bar V^1_i(\a^{-1})) \ge (\frac12 e^{\la_b n})^{m^+} \omega\ ,
$$
proving \eqref{vol} provided $\a$ is sufficiently close to $1$.
\end{proof}

\begin{rmk}
In the proof of (4) above, we have used the implicitly the fact that norm balls $B$ in $E^+$ are \emph{star convex}, i.e., they contain a point (the origin $0$) with the property that any other point  $q$ of $B$ is connected to $B$ by a segment $\ell$ connecting $0$ and $q$. This is the geometric property that enables us to estimate boundaries of norm balls by scaling, as we have done.
\end{rmk}

\section{SRB measures and phase-space observability}

In this section we discuss some consequences of {\bf Theorem A} when applied to
SRB measures with no zero Lyapunov exponents. {\bf Theorem C}, which asserts
that every such SRB measure can be decomposed into at most a countable number
of ergodic SRB measures, is proved is Sect. \ref{subsec:SRBcomponent}. {\bf Theorem D}, which asserts,
in a sense to be clarified, the ``visibility" of SRB measures as a subset of the phase space, 
is proved in Sect. \ref{subsec:globalHolAC}. 

\subsection{Ergodic components of SRB measures}\label{subsec:SRBcomponent}

In addition to the hypotheses (H1)--(H3) at the beginning of this paper, we introduce

\begin{itemize}
\item[(H4)] The Lyapunov exponents of $(f,\mu)$ are nonzero $\mu$-a.e.
\end{itemize}

The aim of this subsection is to prove

\bigskip \noindent
{\bf Theorem C.} {\it Assume (H1)--(H4), and that $\mu$
is an SRB measure. Then
$$
\mu = \sum_{i=1}^\infty c_i \mu_i \qquad \mbox{mod} \ 0
$$
where $c_i \ge 0$ and each $\mu_i$ is an ergodic SRB measure.}

\bigskip
To define SRB measures, we first recall the idea of
{\it stacks of local unstable manifolds} from \cite{BYentropy}. {
As there are no zero Lyapunov exponents and we are} interested 
only in the splitting $E^u=E^+$ and $E^s=E^-$, it suffices to consider
$\Gamma = \Gamma(0; m,p)$. Let $\Gamma_{l_0}$ and $K_n \subset \Gamma$ be
as before. For $\e > 0, \ x_0 \in \Bc$, we write $U(x_0, \e) 
= \{x \in \Bc : |x - x_0| < \e\}$.

\begin{lem}[Lemma 5.5 in \cite{BYentropy} ] \label{lem:uStacksAC}
Let $l_0 \geq 1, n_0 \in \N$, and let $x_0 \in \Gamma_{l_0} \cap K_{n_0}$. Then, there exists $\e_0 > 0$ such that for each $x \in U(x_0, \e_0) \cap \Gamma_{l_0} \cap K_{n_0}$, there is a $C^{1 + \Lip}$ mapping $\Theta^u (x) : B^+_{x_0}(\d l_0^{-3}) \to B^-_{x_0}(\d l_0^{-3}) $ such that $\exp_{x_0} \graph \Theta^u(x) \subset W^u_{x, {\rm loc}}$, $\Lip(\Theta^u(x))\leq \frac{1}{10}$ and $\Lip(d \Theta^u(x)) \leq C_u l_0^2$, where $C_u > 0$ is a constant independent of $\d$. Moreover, the assignment $x \mapsto \Theta^u(x)$ varies continuously in the uniform norm on  $C^0(B^+_{x_0}(\d l_0^{-3}), B^-_{x_0}(\d l_0^{-3}))$.
\end{lem}

An \emph{unstable stack} $\Sc^u$ is a set of the form
\[
\Sc^u = \bigcup_{x \in \bar U} \exp_{x_0} \graph \Theta^u(x)
\]
for some fixed compact $\bar U \subset U(x_0, \e_0) \cap \Gamma_{l_0} \cap K_{n_0}$.
Given $\Sc^u$ with $\mu(\Sc^u) > 0$, let $\eta$ denote the (measurable) partition of $\Sc^u$ into unstable leaves. We consider the \emph{canonical disintegration} $\{\mu_{W}\}_{W \in \eta}$ of $\mu|_{\Sc^u}$ with respect to $\eta$, i.e., for Borel $K \subset \Sc^u$, we have
\[
\mu(K) = \int_{\Sc^u / \eta} \big(  \mu_{W}(K \cap W ) \big)  d \mu^T(W) \, .
\]
Here, $\mu^T$ is the quotient measure on $\Sc^u / \eta$; for details, see \cite{rokhlin}.

\begin{defn}\label{defn:SRBAC} Let $(f,\mu)$ satisfy (H1)--(H4), and assume that
$\la_1 > 0$ $\mu$-a.e. We say $\mu$ is an {\it SRB measure} if for any 
$\Gamma = \Gamma(0; m,p)$ and any unstable stack $\Sc^u$ of positive $\mu$-measure
consisting of leaves through $x \in \bar U \subset \Gamma_{l_0} \cap K_{n_0}$,
the disintegration $\{\mu_W\}_{W \in \eta}$ has the property that for $\mu^T$-almost every $W \in \eta$, 
$\mu_{W}$ is equivalent to $\nu_{W}$, the measure induced on $W$ from volume 
elements in $\Bc$.\footnote{We remark that Definition \ref{defn:SRBAC} is slightly stronger than the definition 
of SRB measures given in \cite{BYentropy}: here we assume not only that $\mu_W$ is absolutely continuous with respect to $\nu_W$ but that the densities are strictly positive $\nu_W$-a.e. 
This definition is more
convenient for us; the results in \cite{BYentropy} hold also under this definition.}
\end{defn}

The following terminology will be useful: Consider a homeomorphism $T$ 
of a compact metric space
$X$ preserving an invariant probability $\nu$. We say a point $x \in X$ is 
\emph{future-generic} with respect to $(T,\nu)$ if for every continuous $\phi : X \to \R$, we have that
\begin{align}\label{eq:forwardTAAC}
\lim_{n \to \infty} \frac1n \sum_{i = 0}^{n-1} \phi \circ T^i(y) = \int \phi \, d \nu \, .
\end{align}
{\it Past genericity} is defined similarly with $T$ replaced by $T^{-1}$,
and we say $T$ is {\it generic} if it is both future and past generic. 
It follows from the Birkhoff Ergodic Theorem that $\nu$-a.e. $x \in X$ is
generic with respect to $(T,\nu)$ if and only if $(T,\nu)$ is ergodic. Furthermore,
by the Ergodic Decomposition Theorem, for any invariant probability $\nu$,
$\nu$-a.e. $x$ is generic with respect to some ergodic measure $\nu_*^x$,
and $\nu = \int \nu_*^x d\nu(x)$.

\begin{proof}[Proof of Theorem C]  We will show that $\mu$ is {\it locally ergodic} 
in the following sense: 
For arbitrary $\Gamma = \Gamma(0;m,p)$ and $l_0, n_0$ 
for which $\mu(\Gamma_{l_0} \cap K_{n_0})>0$, 
it is easy to see that $\Gamma_{l_0} \cap K_{n_0}$ is the union of a countable number 
of positive $\mu$-measure sets $\bar U$, each one of which is small enough 
that it can be used to define both a stack of stable manifolds $\Sc^s_{\bar U}$ (Lemma \ref{lem:stableStacks}) and a stack of unstable leaves 
$\Sc^u_{\bar U}=\cup W$ (see above). We will show that for each such
$\bar U$, there is an ergodic measure  $\mu_* = \mu_*^{\bar U}$ with respect to
which $\mu$-a.e. $x \in \Sc^u_{\bar U}$ is generic.

Let $\bar U$ be fixed. Since $\mu$ is an SRB measure, it follows that for 
$\mu^T$-a.e. $W$ and 
$\nu_W$-a.e. $x \in W$, there is an ergodic measure $\mu_*^x$ with respect to which
$x$ is generic. First we note that if $x,y$ lie in the same $W$, then $\mu_*^x = \mu_*^y$ because orbits
through $x$ and $y$ are backward asymptotic. Thus for $\mu^T$-a.e. $W$ in $\Sc^u$,
there is an ergodic measure $\mu_*^W$ with respect to which $x$ is generic for
$\mu_W$-a.e. $x \in W$. To connect the $\mu_*^W$ for different $W$, observe that by {\bf Theorem A}, either (a) $\nu_W(W \cap \Sc^s)>0$ for every $W$, or 
(b) $\nu_W(W \cap \Sc^s)=0$ for every $W$. Since $\mu(\bar U)>0$ and
$\nu_W$ is absolutely continuous with respect to $\mu_W$, it follows that (a) must hold.
Furthermore, by the equivalence of $\mu_W$ and $\nu_W$ on $\mu^T$-a.e. $W$,
we have that $\mu_*^x$ is defined for $\nu_W$-a.e. $x \in W \cap \Sc^s$.
This together with $\mu_*^x = \mu_*^y$ for $y \in W^s_x$ implies that 
$\mu_*^W=\mu_*^{W'}$ for $\mu^T$-a.e. $W,W'$. 
This common measure $\mu_*^W$ is $\mu_*^{\bar U}$. By the ergodic decomposition
argument above, 
we have that $\mu_*^{\bar U}$ and $\mu$ coincide on $\Sc^u_{\bar U}$ mod zero.

Since a countable union of sets of the form $\bar U$ has 
full $\mu$-measure, it follows that $\mu$ has at most a countable number of 
{ergodic components $\mu_i$, each given by $\mu_*^{\bar U_i}$ for some 
$\bar U_i$. 

It remains to show that each $\mu_i$ is an SRB measure. That is,
we need to verify Definition \ref{defn:SRBAC} for the stack $\Sc^u_{\bar U}$
for every small compact set $\bar U$ with $\mu_i(\bar U)>0$ 
for which the stable/unstable stacks $\Sc^{s/u}_{\bar U}$ are defined as in Lemma \ref{lem:stableStacks}/Lemma \ref{lem:uStacksAC}. We will do so by checking that $\mu_i$ and $\mu$ coincide mod zero on $\Sc^u_{\bar U}$:
Repeating the above arguments, we have that $\mu$-almost every $x \in \Sc^u_{\bar U}$ is generic to an ergodic measure $\mu_*^{\bar U}$, i.e. $\mu|\Sc^u_{\bar U} = 
\mu_*^{\bar U}|\Sc^u_{\bar U}$. That $\mu_*^{\bar U} = \mu_i$ follows from the fact that 
$\mu_i$ is an ergodic component of $\mu$, and $\mu_i(\bar U) > 0$. }
\end{proof}

\subsection{Global holonomy and ``visibility" of SRB measures}\label{subsec:globalHolAC}

Consider the setting in Sect. 2.3 -- Section 3, with the notation and chart systems defined there.
For $x \in \Gamma$, define 
$$
W^{ss}_x := \{y \in \Bc : \limsup_{n \to \infty} \frac1n \log d(f^n x, f^n y) \leq \la^- \} \, .
$$
Because chart sizes shrink more slowly than $\la^-$, it is easy to see that
\begin{equation} \label{eq:characterizeWss}
W^{ss}_x = \bigcup_{n = 0}^\infty f^{-n} \big( W^{ss}_{f^n x, {\rm loc}} \big) \, ,
\end{equation}
where $h_x : \tilde B_x^-(\d_1' l(x)^{-1}) \to \tilde B_x^+(\d_1' l(x)^{-1})$ and
$W^{ss}_{ x, {\rm loc}} = \exp_x \graph h_x$ are as in Theorem \ref{thm:stabManifold}. The sets
$W^{ss}_x$ are {\it global strong stable sets} associated with points $x \in
\Gamma$. In the setting 
under consideration, they are not guaranteed to be immersed submanifolds, as 
$df_x$ is generally not onto and therefore not invertible. The manifold structure of $W^{ss}_x$ can be proved under the following 
assumption, which holds for 
the time-$t$ solution mappings of a broad class of dissipative parabolic PDEs \cite{henry}.

\begin{itemize}
\item[(D)] For any $x \in \Bc$, the operator $df_x$ has dense range in $\Bc$.
\end{itemize}

\begin{prop}[\cite{henry}]\label{prop:inverseImageManifold} Assume that $f$ satisfies assumption (D)
in addition to (H1) - (H3), and let $W$ be an embedded submanifold of codimension $k$. Then, $f^{-1} W$ is an embedded submanifold of codimension $k$.
\end{prop}

We include the proof of Proposition \ref{prop:inverseImageManifold} for completeness.

%

\begin{proof}[Proof of Proposition \ref{prop:inverseImageManifold}]
Fix arbitrary $p \in f^{-1}W$, and let $g : U \to \R^k$ be a $C^1$ submersion
(i.e. $dg$ has full rank) on an open set $U \subset \Bc$ with 
$f(p) \in U$ and for which $g^{-1}(0) = W \cap U$. Observe that $f^{-1}(W \cap U) = (g \circ f)^{-1}(0)$, so it suffices to check that $d(g \circ f)_q = dg_{f q} \circ df_q : \Bc \to \R^k$ has full rank for $q$ in a neighborhood of $p$. Now there exists a $k$-dimensional complement 
$E$ to $\ker (dg_{f p})$ such that $E \subset (df_p)\Bc$, by the dense range assumption. Let $E' \subset \Bc$ be the $k$-dimensional subspace for which $df_p E' = E$. Let 
$V \subset \Bc$ be a small enough neighborhood of $p$ such that $f(V)\subset U$
and the following hold for all $q \in V$: $df_q|E'$ is injective and $d_H(E, df_q(E'))$ is sufficiently small that $dg_{fq}|df_q(E')$
is injective. This implies that $d(g \circ f)$ has full rank on $V$ as desired.
\end{proof}

From Proposition \ref{prop:inverseImageManifold} we immediately obtain the following.
\begin{cor}[Global Strongly Stable Manifold Theorem] \label{cor:wssManifoldThm}
Assume that $f$ satisfies (H1)--(H3) and (D). For any $x \in \Gamma$, $W^{ss}_x$ is an immersed $C^1$ submanifold of $\Bc$ having the same finite codimension as $W^{ss}_{x, {\rm loc}}$.
\end{cor}

Under Assumption (D) then, we may refer to $W^{ss}_x$ as the {\it global strong stable
manifold} at $x$ associated with the rate of convergence $\la^-$.

Corollary \ref{cor:wssManifoldThm} makes possible the extension of local results on absolute continuity
of $W^{ss}$-foliations such as those in Theorem A to holonomy maps along global strong stable manifolds. 
There are many ways to formulate results of this kind, all of which boil down to 
their reduction to local holonomy maps. Here we present one version that has a strong
implication on the ``visibility" of SRB measures. 

For an ergodic measure $\mu$ of $f$,
we define the {\it basin of} $\mu$ to be the set
\begin{align*}
 \Nk(\mu)  := \{x\in \Bc : & \lim_{n \to \infty} d(f^n x, \As) = 0 \, , \text{ and} \\
 & \lim_{n \to \infty} \frac1n \sum_{i = 0}^{n-1} \phi \circ f^i x = \int \phi \, d \mu \ \ \text{ for any } \phi \in C_b^0(\Bc)\}\ ,
\end{align*}
where $C_b^0(\Bc)$ denotes the set of bounded continuous functions on $\Bc$.
The set $\Nk(\mu)$ so defined is a Borel subset of $\Bc$ by an elementary analysis
lemma, a proof of which we have included in the Appendix.

We wish to state next that the basin of an SRB measure occupies a significant subset
of the phase space.
In the absence of a reference measure on $\Bc$ that plays the role of
 Lebesgue measure on $\R^n$, we resort
to the use of finite-dimensional ``probes". For 
a finite dimensional manifold $W \subset \Bc$, the measure $\nu_W$ on $W$ 
induced from volume elements on finite dimensional subspaces
of $\Bc$ is a natural reference measure. 
Theorem D expresses the fact that the basins of ergodic SRB measures
are ``visible" with respect to these reference measures on suitably placed finite-dimensional probes.

\bigskip \noindent
{\bf Theorem D.} {\it In addition to (H1)-(H4) and (D), we assume $\mu$ is an ergodic
SRB measure. Let $W$ be a $C^2$-embedded
disk of dimension $k \ge \dim E^u$. If $W$ meets $W^s_{x_0}$ transversally at one point for some density point $x_0 \in \Gamma$ of $\mu$, 
then $\nu_W(\Nk(\mu))>0$.
}


\begin{proof} Assume first that $k=\dim E^u$.
Then by iterating forward, there is an $N \in \mathbb Z^+$ such that
a component of $f^N(W)$ in the chart at $f^N x_0$ satisfies the condition for
$\Sigma^1$ at the beginning of Section 4 with $E^+=E^u$ and $E^-=E^s$. (This involves proving the analog of what is sometimes referred to as an ``inclination lemma" in finite dimensions; the proof follows from
techniques similar to those used in Section 3 and is omitted.) 
By the fact that $f^N x_0$ is also a density point of $\mu$, it follows from Lemma \ref{lem:uStacksAC}
that it lies in a stack of unstable leaves $\Sc^u$ with $\mu(\Sc^u)>0$. 
{\bf Theorem A} together with the SRB property of $\mu$ then
implies that $\nu_{f^N W}( \Nk(\mu))>0$. Since $f^{N}|_{W}$ is a diffeomorphism with a $C^1$ inverse, we conclude that $\nu_W(\Nk(\mu))>0$ 
as well. 

If $\dim W > \dim E^u$, it is easy to decompose $W$ into a smooth family 
$W = \cup D_\alpha$ where each $D_\alpha$ is a disk having dimension 
$\dim E^u$ and transversal to $W^s_{x_0}$. The argument above applies to 
each $D_\alpha$; we then integrate the result.
\end{proof}

\begin{rmk}
There are many extensions of the notion of ``Lebesgue measure zero'' to the setting of infinite-dimensional spaces; for a survey, see Chapter 6 in \cite{benyaminiLindenstrauss}. {The property possessed by the basin $\Nk(\mu)$ of SRB measures as shown 
in {\bf Theorem D} is stronger than many of these notions. For example, it implies that 
$\Nk(\mu)$ is not of `measure zero' in the framework of prevalence/shyness \cite{SHY} 
(shyness is called Haar null in \cite{benyaminiLindenstrauss}). The proof is
similar to that showing that in $\mathbb R^n$, positive Lebesgue measure sets
are not shy (see \cite{benyaminiLindenstrauss} or \cite{SHY}); 
modifications are left to the reader.}
\end{rmk}

\subsection*{Appendix: A technical lemma }

\newcommand{\Dc}{\mathcal D}

\begin{lem} Let $A \subset Y$ be a compact subset of a metric space $Y$. 
Let $h: Y \to Y$ be a continuous map with $h^{-1}A =A$, and let 
$\nu$ be a Borel probability measure on $A$. Then the basin of $\nu$ (as
defined in Sect.6.2) is a Borel subset of $Y$.
\end{lem}

\begin{proof} The concern is that $C^0_b(Y)$ can be large. As $C^0(A)$ has 
a countable dense subset $\Dc$, it suffices to show
that trajectory averages for $\phi \in C^0(A)$ can be approximated by those for
functions in $\Dc$.
Let $\phi \in C^0_b(Y)$ be given. Fix $\e > 0$ and let $\psi \in \Dc$ be such that
 $\|\phi|_A - \psi\|_{C^0(A)} < \e$. By the Tietze Extension Theorem, $\psi$ has 
a bounded continuous extension $\tilde \psi$ to all of $Y$. For each $i \geq 0$, let $y_i \in A$ 
be such that $d(h^i x, A) = d(h^i x, y_i)$. Then
\begin{eqnarray*}
& \ & |\sum_{i = 0}^{n-1} \phi \circ h^i x - \sum_{i = 0}^{n-1} \tilde \psi \circ h^i x|\\
& \le & \bigg| \sum_{i = 0}^{n-1} \phi \circ h^i x - \sum_{i = 0}^{n-1} \phi (y_i) \bigg|  + \bigg| \sum_{i = 0}^{n-1} \phi(y_i) - \sum_{i = 0}^{n-1} \tilde \psi(y_i) \bigg|  + 
\bigg|  \sum_{i = 0}^{n-1} \tilde \psi(y_i) - \sum_{i = 0}^{n-1} \tilde \psi \circ h^i x \bigg| \ .
\end{eqnarray*}
The middle term is $\leq \e n$. For the first term we use the fact that 
there exists $\hat \d>0$ (depending on $\e$ and $\phi$) 
such that for all $x \in A$ and $y \in Y$ with $d(x,y)<\hat \d$, we have 
$|\phi(x)-\phi(y)|<\e$. The third term is disposed of similarly.
\end{proof}

\bibliography{biblio}
\bibliographystyle{plain}

\end{document}